\documentclass{article}
\usepackage[utf8]{inputenc}
\usepackage{amsmath}
\usepackage{amsthm}
\usepackage{amssymb}
\usepackage{bbm}
\usepackage{comment}
\usepackage{float}
\usepackage{geometry}
\usepackage{graphicx}
\usepackage[usenames,dvipsnames]{xcolor}
\usepackage{hyperref}
\usepackage{dsfont,url}
\usepackage{subcaption}
\usepackage{tikz}

\usetikzlibrary{positioning}
\usetikzlibrary{arrows}
\usetikzlibrary{arrows.meta}
\usetikzlibrary{calc}

\newtheorem{theorem}{Theorem}

\newtheorem{assumption}{Assumption}

\newtheorem{definition}[theorem]{Definition}

\newtheorem{lemma}[theorem]{Lemma}
\newtheorem{claim}[theorem]{Claim}

\newtheorem{remark}{Remark}

\def\qed{\hbox{${\vcenter{\vbox{		 
   \hrule height 0.4pt\hbox{\vrule width 0.5pt height 6pt
   \kern5pt\vrule width 0.5pt}\hrule height 0.4pt}}}$}}

\def\cA{\mathcal A}

\def\cC{\mathcal C}

\def\cG{\mathcal G}

\def\cR{\mathcal R}

\def\bI{\mathbb I}

\def\b0{\mathbb O}

\def\eps{\varepsilon}

\def\t{\tilde}
\let\Re\undefined
\DeclareMathOperator{\Re}{Re}

\def\diag{\mathop{\rm diag}}

\newcommand{\ste}[1]{\textcolor{red}{#1}}

\begin{document}

\title{Global stability of multi-group SAIRS  epidemic  models}
\author{Stefania Ottaviano$^{1,2}$, Mattia Sensi$^{3,4,*}$, Sara Sottile$^5$
\\[1em]
\small $^1$ University of Padua, Dept. of Mathematics “Tullio Levi Civita”, Viale Trieste 63, 35131 Padova, Italy.\\
\small $^2$ University of Trento, Dept.~of Civil, Environmental and Mechanical Engineering, Via Mesiano, 77, 38123 Trento, Italy.\\
\small $^3$MathNeuro Team, Inria at Universit\' e C\^ote d'Azur, 2004 Rte des Lucioles, 06410 Biot, France.\\
\small $^4$Politecnico di Torino, Corso Duca degli Abruzzi 24, 10129 Torino Italy\\
\small $^*$Corresponding author: \texttt{mattia.sensi@polito.it}\\
\small $^5$University of Trento, Dept. of Mathematics, Via Sommarive 14, 38123 Povo - Trento, Italy.}

\maketitle

\begin{abstract}
We study a multi-group SAIRS-type epidemic model with vaccination. The
role of asymptomatic and symptomatic infectious individuals is explicitly considered in the transmission pattern of the disease 
among the groups in which the population is divided. 
This is a natural extension of the homogeneous mixing SAIRS model with vaccination studied in Ottaviano et. al (2021) to a network of communities. We provide a global stability analysis for the model. We determine the value of the basic reproduction number $\mathcal{R}_0$ and prove that the disease-free equilibrium is globally asymptotically stable if $\mathcal{R}_0 < 1$. In the case of the SAIRS model without vaccination, we prove the global asymptotic stability of the disease-free equilibrium also when $\mathcal{R}_0=1$. Moreover, if $\mathcal{R}_0 > 1$, the disease-free equilibrium is unstable and a unique endemic equilibrium exists. First, we investigate the local asymptotic stability of the endemic equilibrium and subsequently its global stability, for two variations of the original model. Last, we provide numerical simulations to compare the epidemic spreading on different networks topologies. 
\bigskip

\smallskip

\end{abstract}

\section{Introduction}
One of the most common assumptions in classic population models is the homogeneity of interactions between individuals, which then happen completely
at random. While such an assumption significantly simplifies the analysis of the models, it can be beneficial to renounce it and to formulate models with more realistic interactions. Heterogeneity in the interactions among the population can depend on many factors. The most common division regards the geographical distinction and the membership to different communities, cities or countries, in which the same infectious disease can have a different behaviour based on the group under study.

The division in groups can also depend on the specific disease under study. For example, individuals can be divided into age groups to study children's diseases, such as measles, mumps or rubella, or can be differentiated by the number of sexual partners for sexually transmitted infections.

Multi-group models can also be useful to study disease transmitted via vectors or multiple hosts, such as Malaria or West-Nile virus. 

The concept of equitable partitions has been used to study networks partitioned into local communities with some regularities in their structure, in the case of SIS and SIRS models \cite{bonaccorsi2015epidemic,ottaviano2019community,ottaviano2021some}, by means of the N-Intertwined Mean-Field approximation \cite{van2011n}. In the aforementioned works, the macroscopic structure of hierarchical networks is described by a quotient
graph and the stability of the endemic equilibrium can be investigated by a lower-dimensional system with respect to the starting one.

Several authors proposed multi-group models to describe the transmission behaviour between different communities or cities, see for example \cite{Lajmanovich1976,huang1992stability,yu2009global}.
In this paper, we assume that each individual 
interacts within a network of relationships, due e.g. to different social or spatial patterns; individuals are hence divided into groups, which are not isolated from one another.

As in the homogeneous mixing case, the stability analysis of the equilibrium points of the system under investigation allows to understand its long-term behaviour and, hence, to obtain some insight into how the prevalence of an endemic disease depends on the parameters of the model \cite{thieme1985local} and, in this case, on the network topology.
However, the problem of existence and global stability, especially for the endemic equilibrium, is often mathematical challenging; unfortunately, for many complex multi-group models it remains an open question, or requires cumbersome conditions \cite{mohapatra2015compartmental}.
In this framework, Guo et al. \cite{guo2006global,guo2008graph} and Li and Shuai \cite{li2010global} developed a graph-theoretic method to find Lyapunov functions 
for some multi-group epidemic models 
 which has recently allowed to obtain 
 various results on the global dynamics of SIRS-type models \cite{Muroya2013,muroya2014further} and SEIRS-type models \cite{fan2018global}. 
 
In this paper, we present a multi-group model, as extension of the SAIRS-type model proposed in \cite{ottaviano2022global}, 
where the role of asymptomatic and symptomatic infectious individuals in the disease transmission has been explicitly considered. Asymptomatic cases often remain unidentified and possibly have more contacts than symptomatic individuals, allowing the virus to circulate widely in the population \cite{day2020covid,oran2020prevalence,oran2021proportion,johansson2021sars}. 
The so-called “silent spreaders” are playing a significant role 
even in the current Covid-19 pandemic and numerous recent papers have considered their contribution in the virus transmission 
(see, e.g., \cite{calvetti2020metapopulation,peirlinck2020visualizing,park2020time,li2020substantial,stella2022role}). 
 However, this contribution 
 has proved relevant also for other communicable diseases, such as influenza, cholera, and shigella \cite{kemper1978effects,nelson2009cholera,stilianakis1998emergence,robinson2013model}.

Although models incorporating asymptomatic individuals already exist in the literature,
they have not been analytically investigated as thoroughly as more famous compartmental models.
Since these types of models have been receiving much more attention lately, we believe they deserve a deeper understanding from a theoretical point of view. Thus, we aim to partially fill this gap and provide a stability analysis of the multi-group system under investigation.

In our model, we denote with $S_i$, $A_i$, $I_i$ and $R_i$, $i=1,\dots,n$, the fraction of Susceptible, Asymptomatic infected, symptomatic Infected and Recovered individuals, respectively, in the $i-$th group, such that $S_i + A_i + I_i + R_i = 1$. We remark that, from here on, we will use the terms community and group interchangeably.

The disease can be transmitted by individuals in the classes $A_i$ and $I_i$, within their group, to the susceptible $S_i$, with transmission rate $\beta^A_{ii}$ and $\beta^I_{ii}$, respectively, but also between different groups: e.g., individuals $A_j$ and $I_j$, belonging to the $j$-th community, may infect susceptible individuals $S_i$ of group $i$ with transmission rate $\beta^A_{ij}$ and $\beta^I_{ij}$, respectively. From the asymptomatic compartment, an individual can either progress to the class of symptomatic
infectious 
or recover without ever developing symptoms. 
We assume that the average time of the symptoms developing, denoted by $1/\alpha$, and the recovery rates from both the infectious compartments, $\delta_A$ and $\delta_I$, do not depend on the community of origin, i.e. these parameters depend only on the disease. Furthermore, the average time to return to the susceptible state, $1/\gamma$, only depends on the specific disease under study, and not on the community to which an individual belongs. The remaining parameters of the model depend on the 
community's membership. First, the proportion of susceptible individuals who receive the vaccine might be different for each group; we denote with $\nu_i$, $i=1,\dots,n$, the proportion of susceptible in the $i-$th group who receive a vaccine-induced temporary immunity. Moreover, $\mu_i$, $i = 1,\dots, n$ represent both the birth rates and the natural death rates in community $i$. Finally, individuals of different communities may have contacts each other, by direct transport, but they never permanently move to another community. Therefore, the total population in each group may only change through births and natural deaths; we do not distinguish between natural deaths and disease-related deaths.

\subsection{Outline}

The paper is organised as follows. In Sec.~\ref{model}, we present the system of equations for the multi-group SAIRS model with vaccination, providing its positive
invariant set. In Sec.~\ref{disease_elimination}, we determine  
the  basic  reproduction  number $\mathcal{R}_0$ and prove that the disease-free equilibrium (DFE) is globally asymptotically stable (GAS) if $\mathcal{R}_0 < 1$ and unstable if $\mathcal{R}_0>1$. Moreover, we prove the GAS of the DFE also in the case $\mathcal{R}_0 = 1$, for the model in which no vaccination is administered to the susceptible individuals. In Sec.~\ref{ex_EE}, we prove the existence and uniqueness of an endemic equilibrium (EE) by a fixed point argument, as in \cite{thieme1985local}, since there is no explicit expression for $\mathcal{R}_0$.
In Sec.~\ref{local_stability}, we provide sufficient conditions for the local asymptotic stability of the EE.
In Sec.~\ref{global_stability}, we discuss the uniform persistence of the disease and we investigate the global asymptotic stability of the EE for two variations of the original model under study. Precisely, in Thm. \ref{thm:GAS_SAIR}, we study the global stability of the SAIR model 
(i.e. $\gamma = 0$) and we prove that the EE is GAS if $\mathcal{R}_0 > 1$. 
In Sec.~\ref{GAS_SAIRS}, we establish sufficient conditions for the GAS of the EE for the SAIRS model (i.e., $\gamma \neq 0$) with vaccination, under the restriction that asymptomatic and symptomatic individuals have the same average recovery 
period, i.e. $\delta_A = \delta_I$. 
The problem of the global stability of the endemic equilibrium in the most general case, i.e. $\delta_A \ne \delta_I$, remains open.
In Sec.~\ref{simulations}, we provide some numerical simulations in which we simulate the evolution of the epidemics in four different structures of community networks. 

\section{The model}\label{model}
The system of ODEs which describes the evolution of the disease in the $i$-th community is the following:
\begin{equation}\label{sairs_net}
\begin{split}
     \frac{d S_i(t)}{dt} &= \mu_i  - \sum_{j=1}^n \bigg(\beta^{A}_{ij} A_j(t) + \beta^{I}_{ij} I_j(t)\bigg)S_i(t) -(\mu_i + \nu_i) S_i(t) +\gamma R_i(t),\\ 
     \frac{d A_i(t)}{dt} &=  \sum_{j=1}^n \bigg(\beta^{A}_{ij} A_j(t) + \beta^{I}_{ij} I_j(t)\bigg)S_i(t) -(\alpha + \delta_A +\mu_i) A_i(t), \\ 
     \frac{d I_i(t)}{dt} &= \alpha A_i(t) - (\delta_{I} + \mu_i)I_i(t), \\ 
     \frac{d R_i(t)}{dt} &=  \delta_A A_i(t) +\delta_I I_i(t) + \nu_i S_i(t) - (\gamma + \mu_i)R_i(t), \qquad \qquad i=1,\dots,n,
     \end{split}
\end{equation}
with initial condition  $(S_1(0), A_1(0), I_1(0), R_1(0),\dots,S_n(0), A_n(0), I_n(0), R_n(0))$ belonging to the set
\begin{equation}\label{gamma_inv}
    \bar \Gamma=\{ (S_1, A_1, I_1, R_1,\dots,S_n,A_n,I_n,R_n) \in \mathbb R_+^{4n}| S_i+ A_i+ I_i+R_i = 1, i=1,\dots,n\},
\end{equation}
where $\mathbb R_+^{4n}$ indicates the non-negative orthant of $\mathbb R^{4n}$. 
The flow diagram representing the interaction among two groups of system (\ref{sairs_net}), as well as their internal dynamics, is given in Figure \ref{fig:SAIRS}.
\begin{figure}[ht!]
			\centering
\tikzset{every picture/.style={line width=0.75pt}} 
\begin{tikzpicture}[x=0.75pt,y=0.75pt,yscale=-1,xscale=1]
\draw   (248.38,87.39) .. controls (248.38,74.58) and (257.97,64.21) .. (269.81,64.21) .. controls (281.64,64.21) and (291.24,74.58) .. (291.24,87.39) .. controls (291.24,100.19) and (281.64,110.57) .. (269.81,110.57) .. controls (257.97,110.57) and (248.38,100.19) .. (248.38,87.39) -- cycle ;
\draw   (388.86,87.39) .. controls (388.86,74.58) and (398.45,64.21) .. (410.29,64.21) .. controls (422.12,64.21) and (431.71,74.58) .. (431.71,87.39) .. controls (431.71,100.19) and (422.12,110.57) .. (410.29,110.57) .. controls (398.45,110.57) and (388.86,100.19) .. (388.86,87.39) -- cycle ;
\draw   (538.86,87.39) .. controls (538.86,74.58) and (548.45,64.21) .. (560.29,64.21) .. controls (572.12,64.21) and (581.71,74.58) .. (581.71,87.39) .. controls (581.71,100.19) and (572.12,110.57) .. (560.29,110.57) .. controls (548.45,110.57) and (538.86,100.19) .. (538.86,87.39) -- cycle ;
\draw    (291.24,87.39) -- (299.57,87.63) -- (385.86,87.4) ;
\draw [shift={(388.86,87.39)}, rotate = 539.8399999999999] [fill={rgb, 255:red, 0; green, 0; blue, 0 }  ][line width=0.08]  [draw opacity=0] (8.93,-4.29) -- (0,0) -- (8.93,4.29) -- cycle    ;
\draw    (431.71,87.39) -- (440.04,87.63) -- (535.86,87.39) ;
\draw [shift={(538.86,87.39)}, rotate = 539.86] [fill={rgb, 255:red, 0; green, 0; blue, 0 }  ][line width=0.08]  [draw opacity=0] (8.93,-4.29) -- (0,0) -- (8.93,4.29) -- cycle    ;
\draw    (269.81,110.57) -- (269.81,138.9) -- (568.62,138.9) -- (568.62,111.91) ;
\draw [shift={(568.62,108.91)}, rotate = 450] [fill={rgb, 255:red, 0; green, 0; blue, 0 }  ][line width=0.08]  [draw opacity=0] (8.93,-4.29) -- (0,0) -- (8.93,4.29) -- cycle    ;
\draw    (80,68.29) -- (80,36.29) -- (568.62,37.16) -- (568.62,63.78) ;
\draw [shift={(568.62,66.78)}, rotate = 270] [fill={rgb, 255:red, 0; green, 0; blue, 0 }  ][line width=0.08]  [draw opacity=0] (8.93,-4.29) -- (0,0) -- (8.93,4.29) -- cycle    ;
\draw    (61,65.29) -- (61,24.29) -- (61,24.29) -- (560,23.29) -- (560.29,64.21) ;
\draw [shift={(61,68.29)}, rotate = 270] [fill={rgb, 255:red, 0; green, 0; blue, 0 }  ][line width=0.08]  [draw opacity=0] (8.93,-4.29) -- (0,0) -- (8.93,4.29) -- cycle    ;
\draw    (290.05,78.37) -- (290.05,78.37) -- (351.43,62.39) ;
\draw [shift={(354.33,61.63)}, rotate = 525.4] [fill={rgb, 255:red, 0; green, 0; blue, 0 }  ][line width=0.08]  [draw opacity=0] (8.93,-4.29) -- (0,0) -- (8.93,4.29) -- cycle    ;
\draw    (431.71,79.66) -- (431.71,79.66) -- (493.1,63.68) ;
\draw [shift={(496,62.92)}, rotate = 525.4] [fill={rgb, 255:red, 0; green, 0; blue, 0 }  ][line width=0.08]  [draw opacity=0] (8.93,-4.29) -- (0,0) -- (8.93,4.29) -- cycle    ;
\draw    (581.71,79.29) -- (581.71,79.29) -- (643.1,63.31) ;
\draw [shift={(646,62.55)}, rotate = 525.4] [fill={rgb, 255:red, 0; green, 0; blue, 0 }  ][line width=0.08]  [draw opacity=0] (8.93,-4.29) -- (0,0) -- (8.93,4.29) -- cycle    ;
\draw    (12,90.29) -- (46.57,89.99) ;
\draw [shift={(49.57,89.96)}, rotate = 539.51] [fill={rgb, 255:red, 0; green, 0; blue, 0 }  ][line width=0.08]  [draw opacity=0] (8.93,-4.29) -- (0,0) -- (8.93,4.29) -- cycle    ;
\draw    (91.05,77.37) -- (246,77.29) ;
\draw [shift={(249,77.29)}, rotate = 539.97] [fill={rgb, 255:red, 0; green, 0; blue, 0 }  ][line width=0.08]  [draw opacity=0] (8.93,-4.29) -- (0,0) -- (8.93,4.29) -- cycle    ;
\draw   (49.57,89.96) .. controls (49.57,77.16) and (59.52,66.78) .. (71.79,66.78) .. controls (84.05,66.78) and (94,77.16) .. (94,89.96) .. controls (94,102.76) and (84.05,113.14) .. (71.79,113.14) .. controls (59.52,113.14) and (49.57,102.76) .. (49.57,89.96) -- cycle ;
\draw    (91.05,77.37) -- (91.05,77.37) -- (152.43,61.39) ;
\draw [shift={(155.33,60.63)}, rotate = 525.4] [fill={rgb, 255:red, 0; green, 0; blue, 0 }  ][line width=0.08]  [draw opacity=0] (8.93,-4.29) -- (0,0) -- (8.93,4.29) -- cycle    ;
\draw [color={rgb, 255:red, 208; green, 2; blue, 27 }  ,draw opacity=1 ]   (91.05,99.37) -- (246,99.29) ;
\draw [shift={(249,99.29)}, rotate = 539.97] [fill={rgb, 255:red, 208; green, 2; blue, 27 }  ,fill opacity=1 ][line width=0.08]  [draw opacity=0] (8.93,-4.29) -- (0,0) -- (8.93,4.29) -- cycle    ;
\draw   (248.38,211.13) .. controls (248.38,223.52) and (257.97,233.56) .. (269.81,233.56) .. controls (281.64,233.56) and (291.24,223.52) .. (291.24,211.13) .. controls (291.24,198.74) and (281.64,188.7) .. (269.81,188.7) .. controls (257.97,188.7) and (248.38,198.74) .. (248.38,211.13) -- cycle ;
\draw   (388.86,211.13) .. controls (388.86,223.52) and (398.45,233.56) .. (410.29,233.56) .. controls (422.12,233.56) and (431.71,223.52) .. (431.71,211.13) .. controls (431.71,198.74) and (422.12,188.7) .. (410.29,188.7) .. controls (398.45,188.7) and (388.86,198.74) .. (388.86,211.13) -- cycle ;
\draw   (538.86,211.13) .. controls (538.86,223.52) and (548.45,233.56) .. (560.29,233.56) .. controls (572.12,233.56) and (581.71,223.52) .. (581.71,211.13) .. controls (581.71,198.74) and (572.12,188.7) .. (560.29,188.7) .. controls (548.45,188.7) and (538.86,198.74) .. (538.86,211.13) -- cycle ;
\draw    (291.24,211.13) -- (299.57,210.89) -- (385.86,211.12) ;
\draw [shift={(388.86,211.13)}, rotate = 180.15] [fill={rgb, 255:red, 0; green, 0; blue, 0 }  ][line width=0.08]  [draw opacity=0] (8.93,-4.29) -- (0,0) -- (8.93,4.29) -- cycle    ;
\draw    (431.71,211.13) -- (440.04,210.89) -- (535.86,211.12) ;
\draw [shift={(538.86,211.13)}, rotate = 180.14] [fill={rgb, 255:red, 0; green, 0; blue, 0 }  ][line width=0.08]  [draw opacity=0] (8.93,-4.29) -- (0,0) -- (8.93,4.29) -- cycle    ;
\draw    (269.81,188.7) -- (269.81,161.29) -- (568.62,161.29) -- (568.62,187.3) ;
\draw [shift={(568.62,190.3)}, rotate = 270] [fill={rgb, 255:red, 0; green, 0; blue, 0 }  ][line width=0.08]  [draw opacity=0] (8.93,-4.29) -- (0,0) -- (8.93,4.29) -- cycle    ;
\draw    (80,229.61) -- (80,260.58) -- (568.62,259.73) -- (568.62,234.07) ;
\draw [shift={(568.62,231.07)}, rotate = 450] [fill={rgb, 255:red, 0; green, 0; blue, 0 }  ][line width=0.08]  [draw opacity=0] (8.93,-4.29) -- (0,0) -- (8.93,4.29) -- cycle    ;
\draw    (61,232.61) -- (61,272.19) -- (61,272.19) -- (560,273.16) -- (560.29,233.56) ;
\draw [shift={(61,229.61)}, rotate = 90] [fill={rgb, 255:red, 0; green, 0; blue, 0 }  ][line width=0.08]  [draw opacity=0] (8.93,-4.29) -- (0,0) -- (8.93,4.29) -- cycle    ;
\draw    (290.05,219.85) -- (290.05,219.85) -- (351.42,235.32) ;
\draw [shift={(354.33,236.05)}, rotate = 194.14] [fill={rgb, 255:red, 0; green, 0; blue, 0 }  ][line width=0.08]  [draw opacity=0] (8.93,-4.29) -- (0,0) -- (8.93,4.29) -- cycle    ;
\draw    (431.71,218.61) -- (431.71,218.61) -- (493.09,234.07) ;
\draw [shift={(496,234.81)}, rotate = 194.14] [fill={rgb, 255:red, 0; green, 0; blue, 0 }  ][line width=0.08]  [draw opacity=0] (8.93,-4.29) -- (0,0) -- (8.93,4.29) -- cycle    ;
\draw    (581.71,218.96) -- (581.71,218.96) -- (643.09,234.43) ;
\draw [shift={(646,235.16)}, rotate = 194.14] [fill={rgb, 255:red, 0; green, 0; blue, 0 }  ][line width=0.08]  [draw opacity=0] (8.93,-4.29) -- (0,0) -- (8.93,4.29) -- cycle    ;
\draw    (12,208.32) -- (46.57,208.61) ;
\draw [shift={(49.57,208.64)}, rotate = 180.48] [fill={rgb, 255:red, 0; green, 0; blue, 0 }  ][line width=0.08]  [draw opacity=0] (8.93,-4.29) -- (0,0) -- (8.93,4.29) -- cycle    ;
\draw    (91.05,220.82) -- (246,220.9) ;
\draw [shift={(249,220.9)}, rotate = 180.03] [fill={rgb, 255:red, 0; green, 0; blue, 0 }  ][line width=0.08]  [draw opacity=0] (8.93,-4.29) -- (0,0) -- (8.93,4.29) -- cycle    ;
\draw   (49.57,208.64) .. controls (49.57,221.03) and (59.52,231.07) .. (71.79,231.07) .. controls (84.05,231.07) and (94,221.03) .. (94,208.64) .. controls (94,196.25) and (84.05,186.21) .. (71.79,186.21) .. controls (59.52,186.21) and (49.57,196.25) .. (49.57,208.64) -- cycle ;
\draw    (91.05,220.82) -- (91.05,220.82) -- (152.42,236.29) ;
\draw [shift={(155.33,237.02)}, rotate = 194.14] [fill={rgb, 255:red, 0; green, 0; blue, 0 }  ][line width=0.08]  [draw opacity=0] (8.93,-4.29) -- (0,0) -- (8.93,4.29) -- cycle    ;
\draw [color={rgb, 255:red, 208; green, 2; blue, 27 }  ,draw opacity=1 ]   (91.05,199.53) -- (246,199.61) ;
\draw [shift={(249,199.62)}, rotate = 180.03] [fill={rgb, 255:red, 208; green, 2; blue, 27 }  ,fill opacity=1 ][line width=0.08]  [draw opacity=0] (8.93,-4.29) -- (0,0) -- (8.93,4.29) -- cycle    ;
\draw [color={rgb, 255:red, 208; green, 2; blue, 27 }  ,draw opacity=1 ] [dash pattern={on 4.5pt off 4.5pt}]  (91.05,99.37) -- (269.81,188.7) ;
\draw [color={rgb, 255:red, 208; green, 2; blue, 27 }  ,draw opacity=1 ] [dash pattern={on 4.5pt off 4.5pt}]  (91.05,199.53) -- (269.81,110.57) ;
\draw [color={rgb, 255:red, 208; green, 2; blue, 27 }  ,draw opacity=1 ] [dash pattern={on 4.5pt off 4.5pt}]  (91.05,99.37) -- (410.29,188.7) ;
\draw [color={rgb, 255:red, 208; green, 2; blue, 27 }  ,draw opacity=1 ] [dash pattern={on 4.5pt off 4.5pt}]  (91.05,199.53) -- (410.29,110.57) ;

\draw (21,73.4) node [anchor=north west][inner sep=0.75pt]  [font=\scriptsize]  {$\mu _{i}$};
\draw (103.6,58.01) node [anchor=north west][inner sep=0.75pt]  [font=\scriptsize,rotate=-347.28]  {${\textstyle \mathnormal{\mu _{i}} S_{i}}$};
\draw (151,59.69) node [anchor=north west][inner sep=0.75pt]  [font=\scriptsize]  {$\left( \beta _{ii}^{A} A_{i} \ +\ \beta _{ii}^{I} I_{i}\right) S_{i}$};
\draw (300.78,60.43) node [anchor=north west][inner sep=0.75pt]  [font=\scriptsize,rotate=-347.28]  {${\textstyle \mathnormal{\mu _{i}} A_{i}}$};
\draw (446.23,60.68) node [anchor=north west][inner sep=0.75pt]  [font=\scriptsize,rotate=-347.28]  {${\textstyle \mathnormal{\mu _{i}} I_{i}}$};
\draw (593.02,61.49) node [anchor=north west][inner sep=0.75pt]  [font=\scriptsize,rotate=-347.28]  {${\textstyle \mathnormal{\mu _{i}} R_{i}}$};
\draw (331.76,89.72) node [anchor=north west][inner sep=0.75pt]  [font=\scriptsize]  {$\alpha $};
\draw (406.19,124.66) node [anchor=north west][inner sep=0.75pt]  [font=\scriptsize]  {$\delta _{A}$};
\draw (477.29,90.87) node [anchor=north west][inner sep=0.75pt]  [font=\scriptsize]  {$\delta _{I}$};
\draw (152,100.69) node [anchor=north west][inner sep=0.75pt]  [font=\scriptsize,color={rgb, 255:red, 208; green, 2; blue, 27 }  ,opacity=1 ]  {$\left( \beta _{ij}^{A} A_{j} \ +\ \beta _{ij}^{I} I_{j}\right) S_{i}$};
\draw (296,10.4) node [anchor=north west][inner sep=0.75pt]  [font=\scriptsize]  {$\gamma $};
\draw (278,37.4) node [anchor=north west][inner sep=0.75pt]  [font=\scriptsize]  {$\nu _{i}$};
\draw (64,80.4) node [anchor=north west][inner sep=0.75pt]    {$S_{i}$};
\draw (260,79.4) node [anchor=north west][inner sep=0.75pt]    {$A_{i}$};
\draw (405,79.4) node [anchor=north west][inner sep=0.75pt]    {$I_{i}$};
\draw (553,79.4) node [anchor=north west][inner sep=0.75pt]    {$R_{i}$};
\draw (19,192.4) node [anchor=north west][inner sep=0.75pt]  [font=\scriptsize]  {$\mu _{j}$};
\draw (64,199.4) node [anchor=north west][inner sep=0.75pt]    {$S_{j}$};
\draw (259,199.4) node [anchor=north west][inner sep=0.75pt]    {$A_{j}$};
\draw (405,200.4) node [anchor=north west][inner sep=0.75pt]    {$I_{j}$};
\draw (552,199.4) node [anchor=north west][inner sep=0.75pt]    {$R_{j}$};
\draw (106.92,226.07) node [anchor=north west][inner sep=0.75pt]  [font=\scriptsize,rotate=-13.29]  {${\textstyle \mathnormal{\mu _{j}} S_{j}}$};
\draw (298.92,224.07) node [anchor=north west][inner sep=0.75pt]  [font=\scriptsize,rotate=-13.29]  {${\textstyle \mathnormal{\mu _{j}} A_{j}}$};
\draw (441.92,223.07) node [anchor=north west][inner sep=0.75pt]  [font=\scriptsize,rotate=-13.29]  {${\textstyle \mathnormal{\mu _{j}} I_{j}}$};
\draw (591.92,224.07) node [anchor=north west][inner sep=0.75pt]  [font=\scriptsize,rotate=-13.29]  {${\textstyle \mathnormal{\mu _{j}} R_{j}}$};
\draw (329.76,196.72) node [anchor=north west][inner sep=0.75pt]  [font=\scriptsize]  {$\alpha $};
\draw (405.19,163.66) node [anchor=north west][inner sep=0.75pt]  [font=\scriptsize]  {$\delta _{A}$};
\draw (481.29,197.87) node [anchor=north west][inner sep=0.75pt]  [font=\scriptsize]  {$\delta _{I}$};
\draw (303,274.4) node [anchor=north west][inner sep=0.75pt]  [font=\scriptsize]  {$\gamma $};
\draw (281,246.4) node [anchor=north west][inner sep=0.75pt]  [font=\scriptsize]  {$\nu _{i}$};
\draw (152,222.69) node [anchor=north west][inner sep=0.75pt]  [font=\scriptsize]  {$\left( \beta _{jj}^{A} A_{j} \ +\ \beta _{jj}^{I} I_{j}\right) S_{j}$};
\draw (158,179.69) node [anchor=north west][inner sep=0.75pt]  [font=\scriptsize,color={rgb, 255:red, 208; green, 2; blue, 27 }  ,opacity=1 ]  {$\left( \beta _{ji}^{A} A_{i} \ +\ \beta _{ji}^{I} I_{i}\right) S_{j}$};
\end{tikzpicture}
\caption{Flow diagram for system (\ref{sairs_net}), depicting the interaction between communities $i$ and $j$, as well as their internal dynamics. The solid lines represent internal dynamics within each group, whereas the dashed lines represent the inter-group influence of infected individuals.}
\label{fig:SAIRS}
\end{figure}
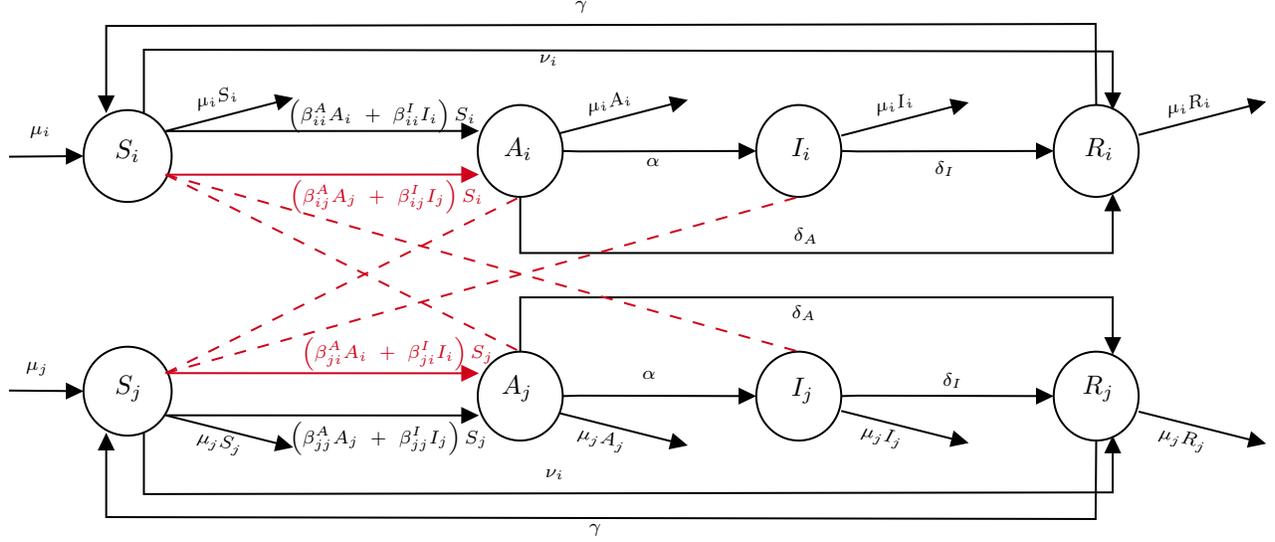\\

Assuming initial conditions in $\bar \Gamma$,  $S_i(t)+A_i(t)+I_i(t)+R_i(t)=1$, for all $t\ge 0$ and $i=1,\dots,n$; hence, system \eqref{sairs_net} is equivalent to the following $3n$-dimensional  dynamical system: 
\begin{equation}\label{sairs3_net}
\begin{split}
    \frac{d S_i(t)}{dt} &= \mu_i -\sum_{j=1}^n \bigg(\beta^{A}_{ij} A_j(t) + \beta^{I}_{ij} I_j(t)\bigg)S_i(t) -(\mu_i + \nu_i +\gamma) S_i(t) + \gamma(1-A_i(t)-I_i(t)), \\ 
      \frac{d A_i(t)}{dt} &=\sum_{j=1}^n \bigg(\beta^{A}_{ij} A_j(t) + \beta^{I}_{ij} I_j(t)\bigg)S_i(t)-(\alpha + \delta_A +\mu_i) A_i(t),  \\
     \frac{d I_i(t)}{dt} &= \alpha A_i(t) - (\delta_I + \mu_i)I_i(t), \qquad \qquad i=1,\dots,n,
\end{split}
\end{equation}
with initial condition  $(S_1(0), A_1(0), I_1(0),\dots,S_n(0), A_n(0), I_n(0))$  belonging to the set 
$$\Gamma=\{ (S_1, A_1, I_1,\dots,S_n,A_n,I_n) \in \mathbb R_+^{3n}| S_i+ A_i+ I_i \leq 1, i=1,\dots,n\}.$$
System \eqref{sairs3_net} can be written in vector notation as

\begin{equation}\label{vectf}
    \frac{dx(t)}{dt}=f(x(t)),
\end{equation}
where $x(t) = (S_1(t), A_1(t), I_1(t),\dots,S_n(t),A_n(t),I_n(t))$ and $f(x(t)) = (f_1(x(t)), f_2(x(t)), \dots, f_{3n}(x(t)))$ is defined according to \eqref{sairs3_net}.\\

\newpage
We make the following assumptions:
\begin{assumption}\label{assump} \mbox{}
\begin{itemize}
    \item{ The matrices $[\beta_{ij}^A]_{i,j=1,\dots,n}$ and $[\beta_{ij}^I]_{i,j=1,\dots,n}$ are irreducible}. This means that every pair of communities is connected by a path.
    \item $\beta^A_{ii} \neq 0$, $\beta^I_{ii} \neq 0, i=1, \ldots, n$. This means that infection can spread within each community.
\end{itemize}
\end{assumption}

\begin{theorem}\label{invset}
  $\Gamma$ is positively invariant for system \eqref{sairs3_net}. That is, for all initial values $x(0) \in \Gamma$, the solution $x(t)$ 
  of \eqref{sairs3_net} will remain in $\Gamma$ for all $t >0$.
 \end{theorem}

\begin{proof}
Let us consider the boundary $\partial \Gamma$, as in \cite[Th. 1]{ottaviano2022global}. It consists of the following hyperplanes:
\begin{align*}
H_{1,i}&=\{(S_1,A_1,I_1,\dots,S_n,A_n,I_n) \in \Gamma \; | \; S_i =0 \}, \\ H_{2,i}&=\{(S_1,A_1,I_1,\dots,S_n,A_n,I_n) \in \Gamma \; | \; A_i=0 \},  \\
H_{3,i} &=\{(S_1,A_1,I_1,\dots,S_n,A_n,I_n) \in \Gamma \; |  \; I_i=0 \},  \\
H_{4,i}&=\{(S_1,A_1,I_1,\dots,S_n,A_n,I_n) \in \Gamma \; | \; S_i+A_i+I_i =1 \}, \qquad \quad  i=1,\dots,n.
\end{align*}
Let us consider $H_{k,1}$, $k=1,2,3,4$.
The outward normal vectors of $H_{1,1}$, $H_{2,1}$, $H_{3,1}$, and $H_{4,1}$ are, respectively 
\begin{align*}
    &\eta_{1,1} = (-1,0,0,\dots,0,0,0),\qquad
    \eta_{2,1} = (0,-1,0,\dots,0,0,0), \\
    &\eta_{3,1} = (0,0,-1,\dots,0,0,0), \qquad
    \eta_{4,1} = (1,1,1,\dots,0,0,0).
\end{align*}
Let 
$x\in H_{k,1}$, $k=1,\dots,4$, and consider the following cases:\\

\emph{Case 1:}  $S_1 = 0$. Then, since $A_1+I_1 \leq 1$, 
 $$\langle f(x), \eta_{1,1}\rangle = -\mu_1 -\gamma(1-A_1-I_1)\leq 0.$$ 

\emph{Case 2:}  $A_1 = 0$. Then, since $S_1 \ge 0$, $A_i \ge 0$, $I_i \ge 0$, $i=2,\dots,n$
 $$\langle f(x), \eta_{2,1}\rangle = - \underbrace{\bigg(\sum_{j=2}^n \beta^{A}_{ij} A_j + \sum_{j=1}^n\beta^{I}_{ij} I_j \bigg)}_{\ge 0} S_1\le 0. $$ 

\emph{Case 3:} $I_1 = 0$. Then, since $A_1 \ge 0$
$$ \langle f(x), \eta_{3,1} \rangle =  -\alpha A_1 \le 0 .$$

\emph{Case 4:} $S_1 + A_1 + I_1 = 1$. Then, since $S_1 \ge 0$, $A_1 \ge 0$, $I_1 \ge 0$
$$ \langle f(x), \eta_{4,1} \rangle = -\nu_1 S_1 -\delta_A A_1 -\delta_I I_1 \leq 0.
$$

The proof for the hyperplanes $H_{k,i}$, $k=1,\dots,4$ and $i=2,\dots,n$ is analogous. 
\end{proof}

\section{Disease Elimination}\label{disease_elimination}

System \eqref{sairs3_net} always admits a disease-free equilibrium, whose expression is:
\begin{equation*}
x_0 = \left(S_{0,1},A_{0,1},I_{0,1},\dots,S_{0,n},A_{0,n}, I_{0,n}\right),
\end{equation*}
where
\begin{equation}\label{DFE_net}
S_{0,i} = \dfrac{\gamma +\mu_i}{\gamma +\mu_i  + \nu_i}, \qquad A_{0,i} = I_{0,i} = 0, \quad\qquad i=1,\dots,n.
\end{equation}
Note that, in general, $S_{0,i} \neq S_{0,j}$ if $i\neq j$.

\begin{lemma}\label{propR0}
Consider the matrix
$$
M_1 = \left(\left(\beta_{ij}^A + \dfrac{\alpha\beta_{ij}^I}{\delta_I + \mu_i}\right)\dfrac{S_{0,i}}{\alpha+\delta_A+\mu_i} \right)_{i,j=1,\dots,n} .
$$
The basic reproduction number $\mathcal{R}_0$ of \eqref{sairs3_net} is 
\begin{equation}\label{R0_net}
    \cR_0 = \rho(M_1) = \rho\left( \left(\left(\beta^A_{ij} + \dfrac{\alpha \beta^{I}_{ij}}{(\delta_I + \mu_i)}\right) \dfrac{\gamma +\mu_i}{(\gamma +\mu_i + \nu_i)(\alpha + \delta_A + \mu_i)} \right)_{i,j=1,\dots,n}  \right),
\end{equation}
where $\rho(M_1)$ is the spectral radius of the matrix $M_1$.
\end{lemma}
\begin{proof}
We shall use the next generation matrix method \cite{van2002repnum} to find $\mathcal{R}_0$. System \eqref{sairs3_net} has $2n$ disease compartments, namely $A_i$ and $I_i$, $i=1,\dots,n$. 
Rearranging the order of the equations such that the disease compartments can be written as $x= (A_1,\dots,A_n,I_1,\dots,I_n)^T$, we can rewrite the corresponding ODEs as
\begin{align*}
\frac{d A_i(t)}{dt} &=  \mathcal{F}_{1_i}(S_i(t),A_i(t),I_i(t)) - \mathcal{V}_{1,i}(S_i(t),A_i(t),I_i(t)) , \\ \nonumber
     \frac{d I_i(t)}{dt} &= \mathcal{F}_{2,i}(S_i(t),A_i(t),I_i(t)) - \mathcal{V}_{2,i}(S_i(t),A_i(t),I_i(t)), \nonumber
\end{align*}
where 
\begin{align*}
&\mathcal{F}_{1,i} = \sum_{i=1}^n \bigg(\beta^A_{ij} A_j(t) + \beta^I_{ij} I_j(t)\bigg)S_i(t), &  \mathcal{V}_{1,i}= (\alpha + \delta_A +\mu_i) A_i(t),\\ \nonumber 
&\mathcal{F}_{2,i}= 0,   &\mathcal{V}_{2,i}= -\alpha A_i(t) + (\delta_I +\mu_i) I_i(t).
\end{align*}
Thus, we obtain
\begin{align}
F  = \left( \begin{matrix} 
\left(\dfrac{\partial \mathcal{F}_{1,i}}{\partial A_j}(x_0)\right)_{i,j=1,\dots,n} & \left(\dfrac{\partial \mathcal{F}_{1,i}}{\partial I_j}(x_0)\right)_{i,j=1,\dots,n} \\ \\
\left(\dfrac{\partial \mathcal{F}_{2,i}}{\partial A_j}(x_0) \right)_{i,j=1,\dots,n} & \left(\dfrac{\partial \mathcal{F}_{2,i}}{\partial I_j}(x_0)\right)_{i,j=1,\dots,n} 
\end{matrix} \right), 
\label{mat_F} \\
V = \left( \begin{matrix}
\left(\dfrac{\partial \mathcal{V}_{1,i}}{\partial A_j}(x_0)\right)_{i,j=1,\dots,n}  & \left(\dfrac{\partial \mathcal{V}_{1,i}}{\partial I_j}(x_0)\right)_{i,j=1,\dots,n} \\ \\
\left(\dfrac{\partial \mathcal{V}_{2,i}}{\partial A_j}(x_0)\right)_{i,j=1,\dots,n}  & \left(\dfrac{\partial \mathcal{V}_{2,i}}{\partial I_j}(x_0)\right)_{i,j=1,\dots,n} 
\end{matrix} \right), \label{mat_V}
\end{align}
which can be written in matrix notation
\begin{align*}
    F = \left(\begin{matrix}
    \Tilde{B}^A & \Tilde{B}^I \\ 0 & 0 
    \end{matrix}\right) \qquad \text{and} \qquad  V =  \left(\begin{matrix}
    (\alpha + \delta_A)\bI + \mu & 0 \\ -\alpha \bI & \delta_I \bI+ \mu 
    \end{matrix}\right),
\end{align*}
where $ (\Tilde{B}^A)_{ij} = \beta_{ij}^A S_{0,i}$, $ (\Tilde{B}^I)_{ij} = \beta_{ij}^I S_{0,i}$, $\mu = \diag{(\mu_1,\dots,\mu_n)}$, and $0$ and $\bI$ are the zero matrix and the identity matrix of order $n$, respectively.
Since $V$ is a block lower triangular matrix, its inverse is the $2n \times 2n$ block matrix:
\begin{equation*}
    V^{-1} = \left(\begin{matrix}
       \diag{\left(\dfrac{1}{\alpha+\delta_A+\mu_i}\right)_{i=1,\dots,n}} & 0 \\
       \diag{\left(\dfrac{\alpha}{(\alpha+\delta_A+\mu_i)(\delta_I+\mu_i)}\right)_{i=1,\dots,n}} & \diag{\left(\dfrac{1}{\delta_I + \mu_i}\right)_{i=1,\dots,n}}
    \end{matrix} \right).
\end{equation*}

The next generation matrix is defined as $M:=FV^{-1}$. By direct calculation, we obtain

\begin{equation}\label{NGM}
    M = \left(
    \begin{matrix}
  \left( \left( \dfrac{\beta^A_{ij}}{\alpha + \delta_A + \mu_i} + \dfrac{\alpha \beta^{I}_{ij}}{(\alpha+\delta_A + \mu_i)(\delta_I + \mu_i)}\right) S_{0,i}\right)_{i,j=1,\dots,n} & \left(\dfrac{\beta^I_{ij}S_{0,i}}{\delta_I + \mu_i}\right)_{i,j=1,\dots,n} \\ 0 & 0 
   \end{matrix}
    \right).
\end{equation}
The basic reproduction number $\cR_0$ is defined as the spectral radius of $M$, denoted by $\rho(M)$, that is $\rho(M) = \max\{\rho(M_1),0 \}$, where

\begin{equation}\label{M1}
M_1 = \left(\left(\beta_{ij}^A + \dfrac{\alpha\beta_{ij}^I}{\delta_I + \mu_i}\right)\dfrac{S_{0,i}}{\alpha+\delta_A+\mu_i} \right)_{i,j=1,\dots,n}.
\end{equation}
As a direct consequence of the Perron Frobenius theorem \cite{horn2012matrix}, $\rho(M_1)>0$. This proves our claim.
\end{proof}

In the following, we present some results to prove the global asymptotic stability of the DFE $x_0$.

Recall that a matrix $M$ is called \textit{non-negative} if each entry is non-negative; we simply write $M \ge0 $ to indicate this.
We use the following results from \cite{van2008r0}:

\begin{lemma}[{\cite[Lemma 2]{van2008r0}}]\label{hyp}
If $F$ is non-negative and $V$ is a non-singular M-matrix, then $\cR_0 = \rho(F V^{-1}) < 1$ if and only if all eigenvalues of $(F - V)$ have negative real parts.
\end{lemma}
Note that the matrices $F$ and $V$ defined in Lemma 
\ref{propR0} satisfy the hypotheses of Lemma \ref{hyp}, thus the following result holds:
\begin{theorem}\label{loc_stability}
 The disease-free equilibrium of \eqref{sairs3_net} is locally asymptotically stable if $\cR_0 < 1$ and unstable if $\cR_0 > 1$.
\end{theorem}
\begin{proof}
See \cite[Theorem 1]{van2008r0}.
\end{proof}


\begin{theorem}\label{DFE_global}
 The disease-free equilibrium $x_0$ of \eqref{sairs3_net} is globally asymptotically stable in $\Gamma$ if $\cR_0 < 1$.
\end{theorem}
\begin{proof}
Let $x(t) = (S_1(t),\dots,S_n(t),A_1(t),\dots,A_n(t),I_1(t)\dots,I_n(t))$ be the solutions of system \eqref{sairs3_net} with initial condition $x(0) \in \Gamma$, in which we have rearranged the order of the equations. In view of Theorem \ref{loc_stability}, it is sufficient to prove that for all $i=1,\dots, n$
\begin{equation*}
    \lim_{t \to \infty} S_i(t) = S_{0,i}, \qquad  \lim_{t \to \infty} A_i(t)= 0, \qquad \text{and} \qquad  \lim_{t \to \infty} I_i(t) = 0,
\end{equation*}
with $S_{0,i}$ as in \eqref{DFE_net}. From the first $n$ equations of \eqref{sairs3_net}, it follows that
\begin{equation*}
\frac{dS_i(t)}{dt} \leq \mu_i+ \gamma  -(\mu_i+\nu_i+\gamma) S_i(t), \qquad i=1,\dots,n.   
\end{equation*}
Thus, $S_{0,i}$ is a global asymptotically stable equilibrium for the comparison equation
\begin{equation*}
\frac{dz_i(t)}{dt}= \mu_i+ \gamma  -(\mu_i+\nu_i+\gamma) z_i(t), \qquad i=1,\dots,n.
\end{equation*}
Then, for any $\eps>0$, there exists $\bar t_i>0$, such that for all $t \geq \bar t_i$, it holds
\begin{equation}\label{S0eps}
    S_i(t) \leq S_{0,i}+ \eps,
\end{equation}
hence 
\begin{equation}\label{limSi}
\limsup_{t \to \infty} S_i(t) \leq S_{0,i}, \qquad i=1,\dots,n.
\end{equation}
Let $\bar t =\max\{t_1,\dots,t_n\}$, then for all $t \geq \bar t$, from \eqref{S0eps} and the remaining $2n$ equations of \eqref{sairs3_net} it follows that
\begin{align*}
    \frac{d A_i(t)}{dt} & \leq \sum_{j=1}^{n}\bigg(\beta^A_{ij} A_j(t) + \beta^I_{ij} I_i(t)\bigg)(S_{0,i} + \eps) -(\alpha + \delta_A +\mu_i) A_i(t),&  \qquad i=1,\dots,n,\\
     \frac{d I_i(t)}{dt} &= \alpha A_i(t) - (\delta_I + \mu_i)I_i(t), &\qquad i=1,\dots,n. \nonumber 
    \end{align*}
Let us now consider the comparison system
\begin{align*}
    \frac{d v_i(t)}{dt} & =\sum_{j=1}^{n}\bigg(\beta^A_{ij} v_j(t) + \beta^I_{ij} u_i(t)\bigg)(S_{0,i} + \eps) -(\alpha + \delta_A +\mu_i) v_i(t),  \\
     \frac{d u_i(t)}{dt} &= \alpha v_i(t) - (\delta_I + \mu)u_i(t), \qquad v_i( \bar t)=A_i( \bar t),\quad u_i( \bar t)=I_i( \bar t), \quad i=1,\dots,n. \nonumber 
    \end{align*}
Let $w = (v_1,\dots,v_n,u_1,\dots,u_n)^T$, then one can rewrite this system as
\begin{equation*}
    \frac{d w(t)}{dt} = (F_\eps - V_\eps) w(t),
\end{equation*}
where $F_\eps$ and $V_\eps$ are the matrices defined in \eqref{mat_F} and \eqref{mat_V}, respectively, evaluated in $x_0(\eps)$ whose components are $ S_{0,i}+\eps$ for $i=1,\dots,n$ and $0$ in the remaining $2n$ components.

Notice that we can choose $\eps > 0$ sufficient small such that $\rho(F_\eps V_\eps^{-1}) < 1$ and then, from Lemma \ref{hyp}, all the eigenvalues of matrix $(F_\eps -V_\eps)$ have negative real parts. It
follows that $\lim_{t\to \infty}w_i(t) = 0$ from any initial conditions in $\Gamma$, from which
$$ \lim_{t\to \infty} A_i(t) = 0 \qquad \text{and} \qquad \lim_{t\to \infty} I_i(t) =0.$$
Thus, for any $\eps >0$, there exists $\bar t_1> 0$
such that, for all 
$t \geq \bar t_1$, we have
$$A_i(t) < \eps \quad \text{and} \quad I_i(t) < \eps, \qquad i=1,\dots,n.$$
From that and the ﬁrst $n$ equations of system \eqref{sairs3_net}, we get that for all $i=1,\dots,n$ and for $t \geq \bar t_1$ 
$$
\frac{d S_i(t)}{dt}\ge  \mu_i -\eps \sum_{j=1}^n(\beta^A_{ij} + \beta^I_{ij})S_i(t) -(\mu_i + \nu_i +\gamma) S_i(t) + \gamma(1-2\eps).
$$
The comparison system
\begin{equation*}
\frac{dz_i(t)}{dt}=\mu_i -\eps \sum_{j=1}^n(\beta^A_{ij} + \beta^I_{ij})z_i(t) -(\mu_i + \nu_i +\gamma) z_i(t) + \gamma(1-2\eps), \qquad i=1,\dots,n,
\end{equation*}
has a globally asymptotically stable equilibrium 
$$z_0 =\left(\frac{\mu_1+\gamma(1-2\eps)}{\eps(\sum_{j=1}^n\beta^A_{1j} + \beta^I_{1j})+ (\mu_1+ \nu_1 +\gamma)},\dots,\frac{\mu_n+\gamma(1-2\eps)}{\eps(\sum_{j=1}^n\beta^A_{nj} + \beta^I_{nj})+(\mu_n+ \nu_n +\gamma)}\right).$$
Thus, we get that for any $\zeta>0$, there exists $\bar t_2> 0$ such that for all $t \geq \bar t_2$,
$$S_i(t) \geq \frac{\mu_i+\gamma(1-2\eps)}{\eps(\sum_{j=1}^n\beta^A_{ij} + \beta^I_{ij})+(\mu_i+ \nu_i +\gamma)} - \zeta, \qquad i=1,\dots,n.$$
This implies that for all $\eps >0$
$$
\liminf_{t\to\infty}S_i(t)\ge\frac{\mu_i+\gamma(1-2\eps)}{\eps(\sum_{j=1}^n\beta^A_{ij} + \beta^I_{ij})+(\mu_i+ \nu_i +\gamma)}, \qquad {i=1,\dots,n}.
$$
Letting $\eps$ go to $0$, we have $\liminf_{t\to\infty}S_i(t)\ge S_{0,i}$ for all $i=1,\dots,n$, which combined with \eqref{limSi} gives us
$$\lim_{t \to \infty} S_i(t)=S_{0,i}, \qquad i=1,\dots,n.$$

\end{proof}

\subsection{SAIRS without vaccination}\label{SAIRS_novacc} 
Let us consider the SAIRS model without vaccination, that is \eqref{sairs3_net} with 
$\nu_i =0$, $i=1, \ldots, n$. From \eqref{R0_net}, the expression of the basic-reproduction number is 
\begin{equation}
    \cR_0 = \rho \left(\left(\left(\beta_{ij}^A + \dfrac{\alpha \beta_{ij}^I}{\delta_I + \mu_i} \right)\dfrac{1}{\alpha + \delta_A + \mu_i}\right)_{i,j=1,\dots,n}   \right),
\end{equation}
and the components of the DFE \eqref{DFE_net} become $S_{0,i} = 1$, $A_{0,i}=I_{0,i} = 0 $, for all $i=1,\dots,n$.

In Theorem \ref{loc_stability} and \ref{DFE_global} we proved that the DFE is globally asymptotically stable if $R_0 < 1$ and unstable if $R_0 > 1$. In the following theorem, which describe the case when we do not have any vaccination, we are able to prove that the DFE is globally asymptotically stable also when $R_0 = 1$. 

\begin{theorem}
The disease-free  equilibrium $x_0$ is  globally  asymptotically  stable in $\Gamma$ for \eqref{sairs3_net} if $\cR_0 \le 1$.
\end{theorem}
\begin{proof}
To prove the statement, we use the method presented in \cite{Shuai2013Lyapunov}. 

Rearranging the order of the equations such that the disease compartments can be written as $x= (A_1,\dots,A_n,I_1,\dots,I_n)$, system \eqref{sairs3_net}, restricted to these compartments,
can be rewritten as:
\begin{equation*}
x' = (F-V) x - f(x,S),
\end{equation*}
where 
\begin{equation*}
f(x,S) = \left( \sum_{j=1}^n \big(\beta^A_{1j} A_j+\beta^I_{1j} I_j\big) (S_{0,1}- S_1), \dots,  \sum_{j=1}^n \big(\beta^A_{nj} A_j+\beta^I_{nj} I_j \big) (S_{0,n}- S_n),0,\dots, 0 \right) \ge 0,
\end{equation*}
and $f(x,S)$ is a vector with non-negative elements for all $(x,S) \in \Gamma$ and $f(x,S_0) = 0$, for all $(x,S_0) \in \Gamma$.

Let $\omega^T$ be the left eigenvector of $M$ corresponding to the eigenvalue $\cR_0$. Note that in our case the irreducibility assumption for $M$ in \cite[Thm 2.2]{Shuai2013Lyapunov} fails. However, we can show that $\omega >0$.
Indeed, let $\omega^T=(\omega_1,\omega_2)$, where $\omega_1$ and $\omega_2$ are both vectors with $n$ components. It is easy to see that $\omega_1$ is the left-eigenvector of the non-negative matrix $M_1$ \eqref{M1} corresponding to its spectral radius $\rho(M_1)=\cR_0$. Since $M_1$ is irreducible and non-negative, it follows by the Perron-Frobenius theorem that $\omega_1>0$.  Moreover, from \eqref{NGM}, let
$$M_2=\left(\dfrac{\beta^I_{ij}S_{0,i}}{\delta_I + \mu_i}\right)_{i,j=1,\dots,n}, $$
then, we have $\omega_1 M_2 = \cR_0 \omega_2$; since $\omega_1 M_2>0$ 
it follows that $\omega_2>0$. Hence, $\omega >0$.
 Now, consider the following Lyapunov function
\begin{equation*}
    Q = \omega^T V^{-1}x.
\end{equation*}
By differentiating $Q$ along the solution of \eqref{sairs3_net}, we obtain
\begin{align*}
    Q' &= \omega^T V^{-1}x'\\
     &= \omega^T V^{-1}\left( F-V\right)x - \omega^T V^{-1} f(x,S)\\
     & = (\cR_0 - 1)\omega^T x -\omega^T V^{-1}f(x,S).  
\end{align*}
Since $\omega^T > 0$, $V^{-1} \ge 0 $ and $f(x,S) \ge 0 $, it follows that $Q' \le (\cR_0 -1) \omega^T x$, Hence, $Q' \le 0$ provided $\cR_0 \le 1$.
Moreover, $Q'=0$ if $x=0$ or $S_i=S_{0,i}$, for all $i=1, \ldots, n$, but this last case still implies $x=0$. It can be verified that the only invariant set where $x=0$ is the singleton $\{x_0\}$. Hence, by LaSalle's invariance principle, the DFE $x_0$ is globally asymptotically stable if $\cR_0\leq1$. \end{proof}

\section{Existence and uniqueness of endemic equilibrium}\label{ex_EE}

 To prove the existence and uniqueness of an endemic equilibrium point, we recall the following definition and theorem from \cite{hethcote1985stability}.

\begin{definition}
A function $F(x):\mathbb{R}^n_+ \rightarrow \mathbb{R}^n_+$ is called \emph{strictly sublinear} if for fixed $x \in (0, \infty)^n$ and fixed $r \in (0,1)^n$, there exists an $\eps>0$ such that $F(rx)\geq (1+\eps)r F(x)$, where $\geq$ denotes the pointwise ordering in $\mathbb R^n$. 
\end{definition}

\begin{theorem}[Thm 2.2 \cite{hethcote1985stability}]\label{thieme}
Let $F(x):\mathbb{R}^n_+ \rightarrow \mathbb{R}^n_+$ be a continuous, monotone nondecreasing, strictly sublinear, and bounded function. Let $F(0)=0$ and $J_F(0)$ exists and be irreducible, where $J_F$ is the Jacobian matrix of $F$. Then $F(x)$ does not have a nontrivial fixed
point on the boundary of $\mathbb R^n_+$. Moreover, $F(x)$ has a positive fixed point if and only if $\rho(J_F(0))>1$. If there is a positive fixed point, then it is unique. 
\end{theorem}
By using the above result, we can prove the following theorem.
\begin{theorem}\label{endemic}
System (\ref{sairs3_net}) admits a unique endemic equilibrium $x^*:=(S_1^*,A_1^*,I_1^*,\dots,S_n^*,A_n^*,I_n^*)$ in $\mathring \Gamma$ if and only if $\cR_0>1$.
\end{theorem}
\begin{proof}
An equilibrium point is a solution of the non linear equations obtained by setting the right-hand side of equations \eqref{sairs3_net} equal to zero. 
Then, the following must hold:
\begin{eqnarray}
     \mu_i -\sum_{j=1}^n \bigg(\beta^{A}_{ij} A_j^* + \beta^{I}_{ij} I_j^*\bigg)S_i^* -(\mu_i + \nu_i +\gamma) S_i^* + \gamma(1-A_i^*-I_i^*)&=0 , \label{endeq1}\\ 
     \sum_{j=1}^n \bigg(\beta^{A}_{ij} A_j^* + \beta^{I}_{ij} I_j^*\bigg)S_i^*-(\alpha + \delta_A +\mu_i) A_i^* &=0,  \label{endeq2}\\
     \alpha A_i^* - (\delta_I + \mu_i)I_i^*&=0, \label{endeq3} 
\end{eqnarray}
for $i=1,2,\dots,n$. By excluding as solution the DFE \eqref{DFE_net}, we assume $A^*_i>0$, for some $1\le i\le n$. From (\ref{endeq3}), we immediately obtain
\begin{equation}\label{IofA}
I_i^* = \dfrac{\alpha}{\delta_I + \mu_i}A_i^*=: K_i A_i^*,
\end{equation}
for all $i=1,2,\dots,n$. Substituting (\ref{IofA}) in (\ref{endeq2}), we obtain
\begin{equation}\label{SofA}
    S_i^*=\dfrac{(\alpha + \delta_A +\mu_i) A_i^*}{\sum_{j=1}^n (\beta^{A}_{ij}  + \beta^{I}_{ij}K_j) A_j^*}.
\end{equation}
By our assumption on $x^*$, the denominator of (\ref{SofA}) is strictly positive. Lastly, substituting (\ref{IofA}) and (\ref{SofA}) into (\ref{endeq1}), we obtain
$$
    \mu_i -(\alpha + \delta_A +\mu_i) A_i^* -(\mu_i + \nu_i +\gamma) \dfrac{(\alpha + \delta_A +\mu_i) A_i^*}{\sum_{j=1}^n (\beta^{A}_{ij}  + \beta^{I}_{ij}K_j) A_j^*} + \gamma(1-(1+K_i)A_i^*)=0,
$$
which can be rearranged to give
$$
 A_i^*=\dfrac{(\mu_i+\gamma)\sum_{j=1}^n (\beta^{A}_{ij}  + \beta^{I}_{ij}K_j) A_j^*}{(\mu_i + \nu_i +\gamma)(\alpha + \delta_A +\mu_i)+(\alpha + \delta_A +\mu_i+\gamma+\gamma K_i)\sum_{j=1}^n (\beta^{A}_{ij}  + \beta^{I}_{ij}K_j) A_j^*}.
$$
We can collect $(\mu_i + \nu_i +\gamma)(\alpha + \delta_A +\mu_i)$ and $(\mu_i+\gamma)$ in both the numerator and denominator, to obtain
\begin{equation}\label{Aendeq}
A_i^*=\dfrac{\sum_{j=1}^n (M_1)_{i,j} A_j^*}{1+(\mu_i+\gamma)^{-1}(\alpha + \delta_A +\mu_i+\gamma+\gamma K_i)\sum_{j=1}^n (M_1)_{i,j} A_j^*},
\end{equation}
with $M_1$ as in (\ref{R0_net}). 

Define a function $H=(h_1,\dots,h_n):\mathbb{R}_+^n \rightarrow \mathbb{R}_+^n$, in the following way:
$$h_i(y)=\dfrac{\sum_{j=1}^n (M_1)_{i,j} y_j}{1+(\mu_i+\gamma)^{-1}(\alpha + \delta_A +\mu_i+\gamma+\gamma K_i)\sum_{j=1}^n (M_1)_{i,j} y_j}, \qquad i=1,2,\dots,n.
$$
Then, since 
$$
\frac{\partial h_j}{\partial y_i}> 0,
$$
for all $i,j=1,2,\dots,n$, $H$ is monotonically increasing in all its components. 
Moreover, $J_H(0)=M_1$ that is a non-negative and irreducible matrix and the function $H(x)$ is bounded and strictly sublinear with
$$\eps=\min_{i} \frac{(1-r)\xi_i \sum_{j=1}^n (M_1)_{i,j} y_j }{1+r\xi_i \sum_{j=1}^n (M_1)_{i,j} y_j},$$
where
$$\xi_i=(\mu_i+\gamma)^{-1}(\alpha + \delta_A +\mu_i+\gamma+\gamma K_i).$$
Thus, by Theorem \ref{thieme}, we have that system \eqref{sairs3_net} has an unique endemic equilibrium in $\mathring \Gamma$.
\end{proof}

\begin{remark}
From Eq. \eqref{IofA} we can note that since $I_i^*<1$, we have that $A_i^*<\frac{\delta_I+\mu_i}{\alpha}$.
\end{remark}

\subsection{Local asymptotic stability}\label{local_stability}

In this section, we 
investigate the local asymptotic stability of the endemic equilibrium.  

\begin{theorem} Assume $\cR_0>1$ and that for any fixed $j$, $\beta^I_{ij}=h_j \beta^A_{ij}$ for all $i=1, \ldots, n,$. Moreover, let us assume that $$\delta_A > \nu_i, \quad \delta_I> \nu_i, \quad \text{and} \quad
(\delta_I-\nu_i)\alpha \leq 2(\mu_i+2\nu_i+\gamma+\delta_I)\sqrt{((\mu_i+\nu_i+\gamma)(\delta_I+\nu_i))}+(\mu_i+2\nu_i+\gamma +\delta_I)^2, $$
 for $i=1, \ldots, n$.
Then, the endemic equilibrium $x^*:=(S_1^*,A_1^*,I_1^*,\dots,S_n^*,A_n^*,I_n^*)$ is local asymptotically stable.
\end{theorem}

\begin{proof}
Usually, the asymptotic local stability of the endemic equilibrium point is studied by linearizing system \eqref{vectf} around that point. However, it is known that the endemic equilibrium is asymptotically stable if the linearized system $\frac{dy}{dt}=J_f(x^*)y$ has no solution of the form $y(t)=Y e^{zt}$ with
$$Y=(U_1, \ldots, U_n, V_1, \ldots, V_n, W_1, \ldots, W_n) \in \mathbb C^{3n},$$
$z \in \mathbb C$, $\Re z \geq 0$, i.e., it means that $zY=J_f(x^*)Y$ with $Y \in \mathbb C^n \setminus \{ 0\}, z \in \mathbb C$ implies $\Re z <0$ \cite{hethcote1985stability,thieme1985local}.
To prove our statement with this strategy, we consider the following system, equivalent to \eqref{vectf}:
$$\frac{dx}{dt}=f(x(t)),$$
where $x(t) = ( A_1(t), I_1(t), R_1(t), \dots, A_n(t),I_n(t),R_n(t))$ and
$f(x(t)) = (f_1(x(t)), f_2(x(t)), \dots, f_{3n}(x(t))),$ with
\begin{align*}
&f(x_1(t))=\sum_{j=1}^n (\beta^A_{ij} A_j(t) + \beta^I_{ij}I_j(t))(1-A_i(t)-I_i(t)-R_i(t))-(\alpha+\mu_i+\delta_A)A_i(t),\\
&f(x_2(t))=\alpha A_i(t)-(\mu_i+\delta_I)I_i(t),\\
&f(x_3(t))=\delta_A A_i(t)+\delta_I I_i(t)+\nu_i(1-A_i(t)-I_i(t)-R_i(t))-(\mu_i+\gamma)R_i(t).
\end{align*}
Now, to prove the asymptotic local stability of $x^*$, we consider the following equations:
\begin{align}\label{UVW}
\begin{split}
zU_i&= (1-A_i^*-I_i^*-R_i^*)\sum_{j=1}^n (\beta^A_{ij} U_j + \beta^I_{ij}V_j)- \sum_{j=1}^n(\beta_{ij}^A A_j^*+\beta_{ij}^I I_j^*)(U_i+V_i+W_i)-(\alpha +\mu_i+\delta_A)U_i,\\  
zV_i&= \alpha U_i -(\delta_I + \mu_i)V_i, \\
zW_i&= (\delta_A-\nu_i)U_i +(\delta_I -\nu_i)V_i-(\mu_i+\nu_i+\gamma)W_i, \qquad i=1, \ldots,n,
\end{split}    
\end{align}
with $U_i,V_i,W_i, z \in \mathbb C$. We proceed by assuming that $\Re z \geq 0$ and showing that this assumption leads to a contradiction.

From the second and third equation of \eqref{UVW}, we have respectively
\begin{equation}\label{Vi}
V_i=\frac{\alpha}{z+\delta_I+\mu_i}U_i := K_i^1(z)U_i,
\end{equation}
and
\begin{equation}\label{Wi}
W_i=\left[\frac{1}{z+\mu_i+\nu_i+\gamma}\left((\delta_A-\nu_i)+\frac{(\delta_I-\nu_i)\alpha}{z+\delta_I+\nu_i}\right)\right]U_i:=K_i^2(z)U_i.
\end{equation}
Now, considering the assumption that fixed $j$, $\beta^I_{ij}=h_j \beta^A_{ij}$ for all $i=1, \ldots, n,$
and replacing \eqref{Vi} and \eqref{Wi} in the first equation of \eqref{UVW}, we obtain
\begin{equation*}
zU_i=S_i^* \sum_{j=1}^n\beta^A_{ij}(1+h_jK^1_j(z))U_j-\left[ \sum_{j=1}^n\beta^A_{ij}(A_j^*+h_j I_j^*)(1+K_i^1(z)+K_i^2(z))+(\alpha+\mu_i+\delta_A)\right]U_i,   
\end{equation*}
from which
\begin{equation}
\left[1+ \frac{1}{\alpha+\mu_i+\delta_A}\left(z+\sum_{j=1}^n\beta^A_{ij}(A_j^*+h_j I_j^*)(1+K_i^1(z)+K_i^2(z)) \right) \right] U_i= \frac{S_i^*}{\alpha+\mu_i+\delta_A}  \sum_{j=1}^n\beta^A_{ij}(1+h_jK^1_j(z))U_j.
\end{equation}
Now, let
\begin{equation*}
\eta_i(z)=\frac{1}{\alpha+\mu_i+\delta_A}\left(z+\sum_{j=1}^n\beta^A_{ij}(A_j^*+h_j I_j^*)(1+K_i^1(z)+K_i^2(z)) \right),
\end{equation*}
and consider the following transformation:
$$U_j=\left(1+\frac{h_j \alpha}{z+\delta_I+\mu_j}\right)^{-1}\left(1+\frac{h_j \alpha}{\delta_I+\mu_j}\right)\tilde U_j.$$
Then, we get
\begin{equation}\label{1eta}
\left(1+ \eta_i(z) \right) \left(1+\frac{h_i \alpha}{z+\delta_I+\mu_i}\right)^{-1}\left(1+\frac{h_i \alpha}{\delta_I+\mu_i}\right)\tilde U_i=   \frac{S_i^*}{\alpha+\mu_i+\delta_A}  \sum_{j=1}^n\beta^A_{ij}\left(1+\frac{h_j\alpha}{\delta_I+\mu_j}\right)\tilde U_j.
\end{equation}
Now, let us note that if $\Re z \geq 0$, then 
\begin{equation}\label{Retildeeta}
\Re \left(\left(1+\frac{h_i \alpha}{z+\delta_I+\mu_i}\right)^{-1}\left(1+\frac{h_i \alpha}{\delta_I+\mu_i}\right)\right) \geq 1.
\end{equation}
Hence, we can rewrite \eqref{1eta} in the following form:
\begin{equation}\label{1eta2}
\left(1+ \eta_i(z) \right)(1+\tilde \eta_i(z)) \tilde U_i=(C\tilde U)_i.   
\end{equation}
where $C=(c_{ij})$ with 
$$c_{ij}=\frac{S_i^*}{\alpha+\mu_i+\delta_A}\sum_{j=1}^n \beta_{ij}^A \left(1+\frac{h_j \alpha}{\delta_I+\mu_j}\right), \qquad i,j=1, \ldots, n.$$
From \eqref{Retildeeta}, we have that $\Re \tilde \eta_i(z) \geq 0$. Moreover, the following claim, whose proof is given in Appendix \ref{appendix}, holds:
\begin{claim}\label{claimRe}
 If $\Re z \geq 0$, then $ \Re \eta_i(z) > 0$. 
\end{claim}
  Now, let us note that $C$ is a non-negative matrix and that 
 $A^*=CA^*$, where $A^*=(A_1^*, \ldots, A^*_n)$. Let 
 $$\eta(z)=\inf\{\Re \eta_i(z), i=1, \ldots,n \}, \qquad \text{and} \qquad |\tilde U|=(|\tilde U_1|, \ldots, |\tilde U_n|),$$ and taking the absolute values in \eqref{1eta2}, we get
 \begin{equation}\label{1etaleq}
    (1+\eta(z))|\tilde U| \leq C|\tilde U|. 
 \end{equation}
 It is easy to verify that if $\Re z \geq 0$, then $\Re \eta_i(z) >0$ for all $i$, hence $\eta(z)>0$. Now, we define $\epsilon$ to be the minimum value for which $|\tilde U| \leq \epsilon A^*$. Since the components of $A^*$ belong to $(0,1)$, $\epsilon < \infty$. Hence, by \eqref{1etaleq}, $(1+\eta(z))|\tilde U|\leq C |\tilde U| \leq \epsilon C A^*= \epsilon A^*$. This inequality contradicts the minimality of $\epsilon$ because $\eta(z)>0$ if $\Re z \geq 0$, thus we can conclude that $\Re z <0$ and the equilibrium is stable.
 \end{proof}

As in \cite{ottaviano2022global}, we conjecture that some, if not all, these technical assumptions could be relaxed, as our numerical simulations suggest. However, the techniques we use in this paper require such assumptions on the parameters in order to reach a result, and multigroup models often require cumbersome hypotheses \cite{SHU20121581,li2009global,SUN2011280}.

 \section{Global stability of the endemic equilibrium}\label{global_stability}
In this section, we first discuss the persistence of the disease, then we investigate the global stability property of the endemic equilibrium for some variations of the original model \eqref{sairs_net}.
\begin{definition}\label{persistence_def1}
System \eqref{sairs3_net} is said to be \emph{uniformly persistent} if there exists a constant $0<\eps<1$ such that any solution $x(t)$ with $x(0) \in \mathring \Gamma $ satisfies 

\begin{equation}\label{pers}
\min \{ \liminf_{t \to \infty} S_i(t), \quad \liminf_{t \to \infty} A_i(t), \quad \liminf_{t \to \infty} I_i(t)\} \geq \eps, \qquad i=1, \ldots, n.
\end{equation}
\end{definition}


\begin{theorem}\label{uniform_per2}
If $\mathcal{R}_0>1$, 
system \eqref{sairs3_net} is uniformly persistent. 
\end{theorem}

\begin{proof}
From Theorem \ref{endemic} we know that DFE $x_0$ is the unique equilibrium of \eqref{sairs3_net} on $\partial \Gamma$, i.e., the largest invariant set on $\partial \Gamma$ is the singleton $\{ x_0\}$, which is isolated. If $\mathcal{R}_0>1$, we know from Theorem \ref{loc_stability} that $x_0$ is unstable. Then, by using \cite[Thm 4.3]{freedman1994uniform}, and similar arguments in \cite[Prop. 3.3]{li1999global}, we can assert that the instability of $x_0$ implies the uniform persistence of \eqref{sairs3_net}.
\end{proof}

\subsection{Global stability of the endemic equilibrium in the SAIR model}\label{GAS_SAIR}

In this section, we study the global asymptotic stability of the endemic equilibrium of the SAIR model, which describes the dynamic of a disease  which confers permanent immunity (i.e. $\gamma = 0$). 
The dynamic of an SAIR model of this type is described by the following system of equations: 
\begin{equation}\label{sair_net}
\begin{split}
    \frac{d S_i(t)}{dt} &= \mu_i -\sum_{j=1}^n \bigg(\beta^{A}_{ij} A_j(t) + \beta^{I}_{ij} I_j(t)\bigg)S_i(t) -(\mu_i + \nu_i ) S_i(t) , \\ 
      \frac{d A_i(t)}{dt} &=\sum_{j=1}^n \bigg(\beta^{A}_{ij} A_j(t) + \beta^{I}_{ij} I_j(t)\bigg)S_i(t)-(\alpha + \delta_A +\mu_i) A_i(t),  \\
     \frac{d I_i(t)}{dt} &= \alpha A_i(t) - (\delta_I + \mu_i)I_i(t), \qquad \qquad i=1,\dots,n.
\end{split}
\end{equation}
The basic reproduction number is derived by substituting $\gamma$ with $0$ in (\ref{R0_net}):
\begin{equation*}
    \cR_0 = \rho\left( \left(\left(\beta^A_{ij} + \dfrac{\alpha \beta^{I}_{ij}}{(\delta_I + \mu_i)}\right) \dfrac{\mu_i}{(\mu_i + \nu_i)(\alpha + \delta_A + \mu_i)} \right)_{i,j=1,\dots,n}  \right).
\end{equation*}
If $\mathcal{R}_0 >1$, system \eqref{sair_net} has a unique equilibrium in $\mathring \Gamma$, which satisfies
\begin{align}
    \mu_i &= \sum_{j=1}^{n} \left(\beta_{ij}^A A^*_j + \beta_{ij}^{I} I^*_j\right)S^*_i + (\mu_i + \nu_i) S_i^*,\label{eq_equations_eeS}\\
    (\alpha+\delta_A + \mu_i)A_i^* &= \sum_{j=1}^{n} \left(\beta_{ij}^A A^*_j + \beta_{ij}^{I} I^*_j\right)S^*_i,\label{eq_equations_eeA}\\
    \alpha A_i^* &= (\delta_I + \mu_i)I_i^*.\label{eq_equations_eeI}
 \end{align}
\begin{theorem}\label{thm:GAS_SAIR}
The endemic equilibrium $x^*$ is globally asymptotically stable in $\mathring \Gamma$ if $\mathcal{R}_0>1$.
\end{theorem}
\begin{proof}
In order to prove the statement, we use a graph-theoretic approach as in \cite{Shuai2013Lyapunov} to establish the existence of a Lyapunov function. Let us define
$$\Tilde{s}_i = \frac{S_i}{S_i^*}, \qquad \Tilde{a}_i = \frac{A_i}{A_i^*}, \qquad \Tilde{i}_i = \frac{I_i}{I_i^*}, $$ 
and $g(x) := x - 1 -\ln(x) \ge 0$ for all $x>0$.
Let
$V_i = V_{i,1} + V_{i,2},$   
where 
$V_{i,1} = S_i^*\cdot g\left( \tilde{s}_i\right), V_{i,2} = A_i^* \cdot g\left(\tilde{a}_i\right),$
and 
$V_{n+i}=I_i^*\cdot g(\t i_i)$, for $i=1, \ldots, n$.

Define $h(x) := -g(x) -1 = -x + \ln(x)$ and note that 
\begin{equation}\label{fun_g}
\left(1- \frac{1}{x}\right)(x-1) = -2 + x + \frac{1}{x} =-1+x-\ln{x}-1+\frac{1}{x}-\ln{\frac{1}{x}}= g(x) + g\left(\frac{1}{x}\right).
\end{equation}
Substituting \eqref{eq_equations_eeS}, \eqref{eq_equations_eeA}, and \eqref{eq_equations_eeI} in \eqref{sair_net}, we obtain
\begin{align*}
        \frac{d S_i(t)}{dt} &= -S_i^* (\mu_i + \nu_i)(\Tilde{s}_i - 1) + \sum_{j=1}^n \left(\beta_{ij}^A\left( A_j^* S_i^* -A_j S_i \right) +\beta_{ij}^I\left( I_j^* S_i^* -I_j S_i \right)\right), \\ 
      \frac{d A_i(t)}{dt} &= \sum_{j=1}^{n} \bigg( \left(\beta_{ij}^A A_j+\beta_{ij}^I I_j\right) S_i - \left(\beta_{ij}^A A^*_j+\beta_{ij}^I I^*_j\right) S^*_i \frac{A_i}{A_i^*}\bigg) ,  \\
     \frac{d I_i(t)}{dt} &= \alpha \left(A_i - A_i^* \frac{I_i}{I_i^*}\right), \qquad \qquad i=1,\dots,n.
\end{align*}
For $i=1, \ldots, n$, differentiating $V_i$ along the solutions of \eqref{sair_net} and using \eqref{fun_g}, we have
\begin{equation}\label{dV1}
    \begin{split}
        \frac{dV_{i,1}}{dt} =& \left(1- \frac{1}{\Tilde{s}_i}\right)   \frac{d S_i(t)}{dt} \\
        =& \left(1- \frac{1}{\Tilde{s}_i}\right) \left[-S_i^* (\mu_i + \nu_i)(\Tilde{s}_i - 1) + \sum_{j=1}^n \left(\beta_{ij}^A\left( A_j^* S_i^* -A_j S_i \right) +\beta_{ij}^I\left( I_j^* S_i^* -I_j S_i \right)\right)\right] \\
        =& \left(1- \frac{1}{\Tilde{s}_i}\right) \left[-S_i^* (\mu_i + \nu_i)(\Tilde{s}_i - 1) + \sum_{j=1}^n \left(\beta_{ij}^A  A_j^* S_i^*\left(1 -\tilde{a}_j \tilde{s}_i \right) +\beta_{ij}^II_j^* S_i^* \left( 1-\tilde{i}_j \tilde{s}_i \right)\right)\right] \\
        =& -S_i^* (\mu_i + \nu_i) \frac{\left( \tilde{s}_i-1\right)^2}{S_i} + \sum_{j=1}^{n} \left(\beta_{ij}^A A_j^* S_i^* \left(1- \tilde{a}_j \tilde{s}_i - \frac{1}{\tilde{s}_i} + \tilde{a}_j\right)\right.\\
        &\left.+\beta_{ij}^I I_j^* S_i^* \left(1- \tilde{i}_j \tilde{s}_i - \frac{1}{\tilde{s}_i} + \tilde{i}_j\right)  \right),
    \end{split}
\end{equation}

\begin{equation}\label{dV2}
      \begin{split}
        \frac{dV_{i,2}}{dt} &= \left(1- \frac{1}{\Tilde{a}_i}\right)   \frac{d A_i(t)}{dt} \\
        &= \left(1- \frac{1}{\Tilde{a}_i}\right) \left[ \sum_{j=1}^{n} \bigg( \left(\beta_{ij}^A A_j+\beta_{ij}^I I_j\right) S_i - \left(\beta_{ij}^A A^*_j+\beta_{ij}^I I^*_j\right) S^*_i \frac{A_i}{A_i^*}\bigg) \right]\\
        &= \left(1- \frac{1}{\Tilde{a}_i}\right)  \left[\sum_{j=1}^{n} \bigg( \beta_{ij}^A A_j^* S_i^* \left( \tilde{a}_j \tilde{s}_i - \tilde{a}_i\right) + \beta_{ij}^I I_j^* S_i^* \left(\tilde{i}_j \tilde{s}_i - \tilde{a}_i\right) \bigg) \right] \\
        &= \sum_{j=1}^{n} \left(\beta_{ij}^A A_j^* S_i^* \left(\tilde{a}_j \tilde{s}_i - \tilde{a}_i - \frac{\tilde{a}_j \tilde{s}_i }{\tilde{a}_i} + 1 \right) + \beta_{ij}^I I_j^* S_i^* \left(\tilde{i}_j \tilde{s}_i - \tilde{a}_i - \frac{\tilde{i}_j \tilde{s}_i }{\tilde{a}_i} + 1 \right)\right) ,
    \end{split}
\end{equation}
Thus, from \eqref{dV1} and \eqref{dV2}, we obtain
\begin{equation}
    \begin{split}
        \frac{dV_i}{dt} \leq & \sum_{j=1}^n \left(\beta_{ij}^A A_j^* S_i^*\left(2 -\frac{1}{\tilde{s}_i} + \tilde{a}_j - \tilde{a}_i - \frac{\tilde{a}_j\tilde{s}_i}{\tilde{a}_i} \right) + \beta_{ij}^I I_j^* S_i^* \left(2 -\frac{1}{\tilde{s}_i} + \tilde{i}_j - \tilde{a}_i - \frac{\tilde{i}_j\tilde{s}_i}{\tilde{a}_i} \right) \right) 
    \end{split}
\end{equation}
Using the fact that $1- x \leq - \ln(x)$, we can write
\begin{equation*}
\begin{split}
2 -\frac{1}{\tilde{s}_i} + \tilde{a}_j - \tilde{a}_i - \frac{\tilde{a}_j\tilde{s}_i}{\tilde{a}_i} &\leq \tilde{a}_j - \tilde{a_i} -\ln\left(\frac{1}{\tilde{s}_i}\right) - \ln\left(\frac{\tilde{a}_j\tilde{s}_i}{\tilde{a}_i}\right)= h(\tilde{a}_i) - h(\tilde{a}_j),\\
2 -\frac{1}{\tilde{s}_i} + \tilde{i}_j - \tilde{a}_i - \frac{\tilde{i}_j\tilde{s}_i}{\tilde{a}_i} &\leq \tilde{i}_j - \tilde{a_i} -\ln\left(\frac{1}{\tilde{s}_i}\right) - \ln\left(\frac{\tilde{i}_j\tilde{s}_i}{\tilde{a}_i}\right)= h(\tilde{a}_i) - h(\tilde{i}_j),\\
\end{split}
\end{equation*}
Thus, we obtain
\begin{equation*}
    \begin{split}
        \frac{dV_i}{dt} &\le \sum_{j=1}^n \left(\beta_{ij}^A A_j^* S_i^* (h(\tilde{a}_i) - h(\tilde{a}_j)) + \beta_{ij}^I I_j^* S_i^* (h(\tilde{a}_i) - h(\tilde{i_j})) \right) =: \sum_{j=1}^{2n} \tilde{\beta}_{ij} G_{i,j},
    \end{split}
\end{equation*}
    where 
    \begin{equation*}
    \tilde{\beta}_{ij} = \left\lbrace \begin{split}
        \beta_{ij}^A A_j^* S_i^*, & \quad 1\leq j \leq n,\\
        \beta_{i\;{j-n}}^I I_{j-n}^* S_i^*, & \quad n+1 \leq j \leq 2n,
    \end{split} \right. \qquad \text{and} \qquad G_{ij} = \left\lbrace \begin{split}
        h(\tilde{a}_i) - h(\tilde{a}_j), & \quad 1\leq j \leq n,\\
        h(\tilde{a}_i) - h(\tilde{i}_{j-n}), & \quad n+1 \leq j \leq 2n.
    \end{split} \right.
    \end{equation*}
Moreover, for all $i=1, \ldots, n$    
\begin{equation}\label{dV3}
      \begin{split}
       \frac{dV_{n+i}}{dt} &=\left(1- \frac{1}{\Tilde{i}_i}\right)  \frac{dI_{i}}{dt}=  \alpha \left(1- \frac{1}{\Tilde{i}_i}\right) \left[ A_i - A_i^* \frac{I_i}{I_i^*}\right] \\
        &=  \alpha A_i^*\left(1- \frac{1}{\Tilde{i}_i}\right) \left(\tilde{a}_i - \tilde{i}_i \right)=  \alpha A_i^* \left( \tilde{a}_i - \tilde{i}_i - \frac{\tilde{a}_i}{\tilde{i}_i}  + 1 \right),
    \end{split}
\end{equation}    
and again, using the fact that $1- x \leq - \ln(x)$, we have

    $$1 + \tilde{a}_i - \tilde{i}_i - \frac{\tilde{a}_i}{\tilde{i}_i}  \leq \tilde{a}_i - \tilde{i}_i - \ln\left(\frac{\tilde{a}_i}{\tilde{i}_i}\right) = h(\tilde{i}_i) - h(\tilde{a}_i).$$
Thus,
\begin{equation}\label{dV3leq}
\frac{dV_{n+i}}{dt} \leq \alpha A_i^* (h(\tilde{i}_i) - h(\tilde{a}_i)) =: \t \beta_{n+i\;i} G_{n+i\;i}.
\end{equation}

\begin{figure}
    \centering

\tikzset{every picture/.style={line width=0.75pt}} 

\begin{tikzpicture}[x=0.75pt,y=0.75pt,yscale=-1,xscale=1]

\draw   (146,76) .. controls (146,62.19) and (157.19,51) .. (171,51) .. controls (184.81,51) and (196,62.19) .. (196,76) .. controls (196,89.81) and (184.81,101) .. (171,101) .. controls (157.19,101) and (146,89.81) .. (146,76) -- cycle ;
\draw   (146,176) .. controls (146,162.19) and (157.19,151) .. (171,151) .. controls (184.81,151) and (196,162.19) .. (196,176) .. controls (196,189.81) and (184.81,201) .. (171,201) .. controls (157.19,201) and (146,189.81) .. (146,176) -- cycle ;
\draw   (303,76.14) .. controls (303,62.34) and (314.19,51.14) .. (328,51.14) .. controls (341.81,51.14) and (353,62.34) .. (353,76.14) .. controls (353,89.95) and (341.81,101.14) .. (328,101.14) .. controls (314.19,101.14) and (303,89.95) .. (303,76.14) -- cycle ;
\draw   (303,176.14) .. controls (303,162.34) and (314.19,151.14) .. (328,151.14) .. controls (341.81,151.14) and (353,162.34) .. (353,176.14) .. controls (353,189.95) and (341.81,201.14) .. (328,201.14) .. controls (314.19,201.14) and (303,189.95) .. (303,176.14) -- cycle ;
\draw    (328,201.14) .. controls (286.42,223.91) and (219.36,223.16) .. (172.42,201.66) ;
\draw [shift={(171,201)}, rotate = 25.23] [color={rgb, 255:red, 0; green, 0; blue, 0 }  ][line width=0.75]    (10.93,-3.29) .. controls (6.95,-1.4) and (3.31,-0.3) .. (0,0) .. controls (3.31,0.3) and (6.95,1.4) .. (10.93,3.29)   ;
\draw    (328,51.14) .. controls (289.39,25.4) and (211.58,26.13) .. (172.18,50.26) ;
\draw [shift={(171,51)}, rotate = 327.49] [color={rgb, 255:red, 0; green, 0; blue, 0 }  ][line width=0.75]    (10.93,-3.29) .. controls (6.95,-1.4) and (3.31,-0.3) .. (0,0) .. controls (3.31,0.3) and (6.95,1.4) .. (10.93,3.29)   ;
\draw    (146,76) .. controls (130.32,116.32) and (134.81,149.78) .. (145.35,174.5) ;
\draw [shift={(146,176)}, rotate = 246.13] [color={rgb, 255:red, 0; green, 0; blue, 0 }  ][line width=0.75]    (10.93,-3.29) .. controls (6.95,-1.4) and (3.31,-0.3) .. (0,0) .. controls (3.31,0.3) and (6.95,1.4) .. (10.93,3.29)   ;
\draw    (196,76) -- (301,76.14) ;
\draw [shift={(303,76.14)}, rotate = 180.08] [color={rgb, 255:red, 0; green, 0; blue, 0 }  ][line width=0.75]    (10.93,-3.29) .. controls (6.95,-1.4) and (3.31,-0.3) .. (0,0) .. controls (3.31,0.3) and (6.95,1.4) .. (10.93,3.29)   ;
\draw    (196,176) -- (301,176.14) ;
\draw [shift={(303,176.14)}, rotate = 180.08] [color={rgb, 255:red, 0; green, 0; blue, 0 }  ][line width=0.75]    (10.93,-3.29) .. controls (6.95,-1.4) and (3.31,-0.3) .. (0,0) .. controls (3.31,0.3) and (6.95,1.4) .. (10.93,3.29)   ;
\draw    (171,151) -- (171,103) ;
\draw [shift={(171,101)}, rotate = 90] [color={rgb, 255:red, 0; green, 0; blue, 0 }  ][line width=0.75]    (10.93,-3.29) .. controls (6.95,-1.4) and (3.31,-0.3) .. (0,0) .. controls (3.31,0.3) and (6.95,1.4) .. (10.93,3.29)   ;
\draw    (303,76.14) .. controls (279.24,134.55) and (252.54,165.52) .. (197.67,175.7) ;
\draw [shift={(196,176)}, rotate = 350.02] [color={rgb, 255:red, 0; green, 0; blue, 0 }  ][line width=0.75]    (10.93,-3.29) .. controls (6.95,-1.4) and (3.31,-0.3) .. (0,0) .. controls (3.31,0.3) and (6.95,1.4) .. (10.93,3.29)   ;
\draw    (303,176.14) .. controls (299.04,87.04) and (243.13,88.11) .. (197.38,76.36) ;
\draw [shift={(196,76)}, rotate = 14.79] [color={rgb, 255:red, 0; green, 0; blue, 0 }  ][line width=0.75]    (10.93,-3.29) .. controls (6.95,-1.4) and (3.31,-0.3) .. (0,0) .. controls (3.31,0.3) and (6.95,1.4) .. (10.93,3.29)   ;

\draw (168,66.4) node [anchor=north west][inner sep=0.75pt]  [font=\small]  {$i$};
\draw (167,166.4) node [anchor=north west][inner sep=0.75pt]  [font=\small]  {$j$};
\draw (311,67.4) node [anchor=north west][inner sep=0.75pt]  [font=\small]  {$i\ +\ n$};
\draw (311,166.4) node [anchor=north west][inner sep=0.75pt]  [font=\small]  {$j\ +\ n$};
\draw (137.76,125.98) node [anchor=north west][inner sep=0.75pt]  [font=\scriptsize]  {$\tilde{\beta}_{ji}$};
\draw (174.39,125.98) node [anchor=north west][inner sep=1pt]  [font=\scriptsize]  {$\tilde{\beta}_{ij}$};
\draw (232.23,35.09) node [anchor=north west][inner sep=0.5pt]  [font=\scriptsize,rotate=-0.9]  {$\tilde{\beta}_{i\;i+n}$};
\draw (240.23,61.09) node [anchor=north west][inner sep=0.75pt]  [font=\scriptsize,rotate=-0.9]  {$\tilde{\beta}_{i+n\;i}$};
\draw (225.23,199.09) node [anchor=north west][inner sep=0.75pt]  [font=\scriptsize,rotate=-0.9]  {$\tilde{\beta}_{j\;j+n}$};
\draw (240.23,179.09) node [anchor=north west][inner sep=0.75pt]  [font=\scriptsize,rotate=-0.9]  {$\tilde{\beta}_{j+n\;j}$};
\draw (222.32,150.48) node [anchor=north west][inner sep=0.75pt]  [font=\scriptsize,rotate=-330.89]  {$\tilde{\beta}_{j\;i+n}$};
\draw (226.71,85.82) node [anchor=north west][inner sep=0.75pt]  [font=\scriptsize,rotate=-14.82]  {$\tilde{\beta}_{i\;j+n}$};

\end{tikzpicture}

\caption{The weighted digraph $(G,\tilde B)$ constructed for system \eqref{sair_net}.}
\label{fig:digraph}
\end{figure}
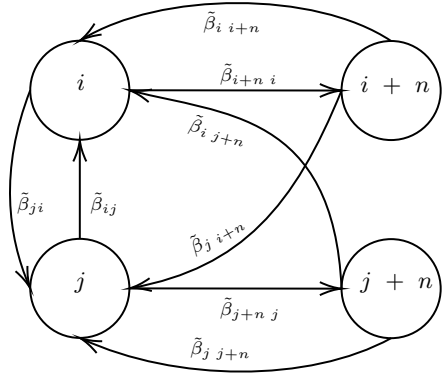

We can construct a weighted digraph $\mathcal{G}$, associated with the weight matrix $\tilde{B} = (\tilde{\beta}_{ij})_{i,j=1,\dots,2n}$, with $\tilde{\beta}_{ij} > 0$ as defined above and zero otherwise; see Figure \ref{fig:digraph}. Let us note that, from Assumption \ref{assump}, the digraph $(\mathcal{G},\tilde{B})$ is strongly connected. 
Since $G_{i\;n+j}+G_{n+j\;j}=-\t a_i +\ln (\t a_i)+\t a_j -\ln (\t a_j)=G_{ij}$, $i,j=1, \ldots, n$, it can be verified that each directed cycle $\cC$ of $(\mathcal{G},\tilde{B})$ has $\sum _{(s,r) \in \mathcal{E}(\mathcal{C})} G_{rs}=0$, where $\mathcal{E}(\mathcal{C})$ denotes the arc set of the directed cycle $\cC$.
Thus, the assumptions of \cite[Theorem 3.5]{Shuai2013Lyapunov} hold, hence the function
$$ V = \sum_{i=1}^n \left(c_i V_i+ c_{n+i} V_{n+i}\right),$$
for constants $c_i>0$ defined as in \cite[Prop. 3.1]{Shuai2013Lyapunov}, satisfies $\frac{dV}{dt} \leq 0$, meaning that $V$
is a Lyapunov function for system \eqref{sair_net}.  
It can be verified that the largest compact invariant set in which $\frac{dV}{dt} = 0$ is the singleton $\{x^*\}$. Hence, our claim follows by LaSalle’s Invariance Principle \cite{la1976stability}.
\end{proof}

\begin{remark}
We observe that the proof of Theorem \ref{thm:GAS_SAIR} also holds for the case $\delta_A=0$ in system \eqref{sairs3_net}. That is to say, for a model with two stages of infection $I_1$ and $I_2$, in which from the first class of infection one passes to the second at the rate $\alpha$ and one cannot directly pass into the compartment of recovered individuals. Then, from the second stage of infection, one can recover at the rate $\delta_{I_2}$. It is known that, if $\alpha=\delta_{I_2}$, the length of the infectious period follows a gamma distribution; otherwise, the resulting distribution is not a standard one.
Moreover, we remark that Theorem \ref{thm:GAS_SAIR} only requires $\mathcal{R}_0>1$, and no additional conditions on the parameters, despite the complexity of the model under study. Models with multiple infected compartments have been studied, e.g., in  \cite{moghadas2002global,sherborne2016compact,wang2014global} . 
\end{remark}

\subsection{Global stability of the SAIRS model when $\delta_A = \delta_I =: \delta$}\label{GAS_SAIRS}

In the $\delta_A = \delta_I =: \delta$ case, from (\ref{R0_net}) we have
$$\cR_0=\rho\left( \left(\left(\beta^A_{ij} + \dfrac{\alpha \beta^{I}_{ij}}{(\delta + \mu_i)}\right) \dfrac{\gamma +\mu_i}{(\gamma +\mu_i + \nu_i)(\alpha + \delta + \mu_i)} \right)_{i,j=1,\dots,n}  \right).$$
If $\mathcal{R}_0 >1$, system \eqref{sairs_net} with $\delta_A=\delta_I=:\delta$ has a unique equilibrium in $\mathring \Gamma$, which satisfies
\begin{align}
\begin{split}\label{equlibrium_eq}
    \mu_i &= \sum_{j=1}^{n} \left(\beta_{ij}^A A^*_j + \beta_{ij}^{I} I^*_j\right)S^*_i + (\mu_i + \nu_i) S_i^* -\gamma R_i^*,\\
    (\alpha+\delta + \mu_i)A_i^* &= \sum_{j=1}^{n}\left(\beta_{ij}^A A^*_j + \beta_{ij}^{I} I^*_j\right)S^*_i,\\
    \alpha A_i^* &= (\delta + \mu_i)I_i^*,\\
    \nu_i S_i^*&= -\delta(A_i^*+I_i^*)+(\gamma +\mu_i)R_i^*.
    \end{split}
\end{align}

\begin{theorem}
Assume that ${(\mu_i+\nu_i)}S_i^* \geq {\gamma R_i^*}$ and $\delta > \nu_i$, for each $i=1, \ldots, n$. Then, the endemic equilibrium $x^*$ is globally asymptotically stable in $\mathring \Gamma$ if $\mathcal{R}_0>1$.
\end{theorem}

\begin{proof}
Let $\t s_i, \t a_i, \t i_i$, $V_i$, and $V_{n+i}$ as in Theorem \ref{thm:GAS_SAIR}. Let us define {$\t r_i=\frac{R_i}{R_i^*}$} and 
$$W_i=\frac{\gamma}{S_i^*(\delta-\nu_i)}\frac{(R_i-R_i^*)^2}{2}, \qquad i=1, \ldots,n. $$
By using equations $\eqref{equlibrium_eq}$, and differentiating along the solution of \eqref{sairs_net} with $\delta_A = \delta_I =: \delta$, we obtain

\begin{equation}\label{dV1new}
    \begin{split}
        \frac{dV_{i,1}}{dt} =& \left(1- \frac{1}{\Tilde{s}_i}\right)   \frac{d S_i(t)}{dt} \\
        =& \left(1- \frac{1}{\Tilde{s}_i}\right) \left[-S_i^* (\mu_i + \nu_i)(\Tilde{s}_i - 1) + \gamma R_i^*(\t r_i-1) + \sum_{j=1}^n \left(\beta_{ij}^A  A_j^* S_i^*\left(1 -\tilde{a}_j \tilde{s}_i \right) +\beta_{ij}^II_j^* S_i^* \left( 1-\tilde{i}_j \tilde{s}_i \right)\right)\right] \\
        =& -S_i^* (\mu_i + \nu_i) \left( 1- \frac{1}{\Tilde{s}_i}  \right) (\t s_i-1)+ \gamma R_i^*\left( 1- \frac{1}{\Tilde{s}_i}  \right) (\t r_i-1) + \sum_{j=1}^{n} \left(\beta_{ij}^A A_j^* S_i^* \left(1- \tilde{a}_j \tilde{s}_i - \frac{1}{\tilde{s}_i} + \tilde{a}_j\right)\right.\\
        &\left.+\beta_{ij}^I I_j^* S_i^* \left(1- \tilde{i}_j \tilde{s}_i - \frac{1}{\tilde{s}_i} + \tilde{i}_j\right)  \right),
    \end{split}
\end{equation}
and the derivatives $\frac{dV_{i,2}}{dt}$ and $\frac{dV_{n+i}}{dt}$ as in \eqref{dV2} and \eqref{dV3}, respectively. Moreover, 
 \begin{align}
 \begin{split}
\frac{d W_i}{dt}&= \frac{\gamma}{S_i^*(\delta-\nu_i)}(R_i-R_i^*) \frac{dR_i}{dt}\\
 &=\frac{\gamma}{S_i^*(\delta-\nu_i)}(R_i-R_i^*)\left[ \delta(A_i-A_i^*+I_i-I_i^*) + \nu_i (S_i-S_i^*)-(\gamma+\mu_i)(R_i-R_i^*)\right]\\
&= \frac{\gamma}{S_i^*(\delta-\nu_i)}(R_i-R_i^*)\left[ \delta(S^*_i-S_i+R^*_i-R_i) + \nu_i (S_i-S_i^*)-(\gamma+\mu_i)(R_i-R_i^*)\right]\\
&= \frac{\gamma}{S_i^*(\delta-\nu_i)} R_i^* S_i^* (\nu_i-\delta) (\t s_i-1)(\t r_i-1)-(\gamma+\mu_i +\delta)R_i^*(\tilde{r}_i-1)^2,
 \end{split}
 \end{align}
 by assumption $\delta > \nu_i$, thus
 \begin{equation}\label{dWleq}
  \frac{d W_i}{dt} \leq - \gamma R_i^* (\t s_i-1)(\t r_i-1).  
 \end{equation}
 Let us consider the weighted digraph $\cG$, the weight matrix $\t B$, and the functions $G_{i,j}$, for $i,j=1, \ldots, 2n$ defined as in Theorem \ref{thm:GAS_SAIR}. Consider the following function:
$$V=\sum_{i=1}^n  \left(c_i V_i + c_{n+i} V_{n+i}\right)+ \sum_{i=1}^n c_i W_i, $$
where the constant $c_i>0$ are defined as in \cite[Prop. 3.1]{Shuai2013Lyapunov}. Then, by following similar steps as in Theorem \ref{thm:GAS_SAIR} and from \eqref{dWleq}, we obtain
\begin{align}
\frac{dV}{dt} &\leq \sum_{i=1}^{2n} \sum_{j=1}^{2n} c_i \t \beta_{ij} G_{i,j} - \sum_{i=1}^n c_i  (\mu_i+\nu_i) S_i^* \left(1-\frac{1}{\t s}\right)(\t s_i-1)+
\sum_{i=1}^n c_i \gamma R_i^* (\t {r}_i-1) \left[\left(1-\frac{1}{\t s_i}\right) -(\t s_i-1) \right]\\\nonumber
&=\sum_{i=1}^{2n} \sum_{j=1}^{2n} c_i \t \beta_{ij} G_{i,j} - \sum_{i=1}^n c_i  (\mu_i+\nu_i) S_i^* \left(1-\frac{1}{\t s}\right)(\t s_i-1)+
\sum_{i=1}^n c_i \gamma R_i^* (\t r_i-1) \left(1-\frac{1}{\t s_i}\right)(1-\t s_i)\\\nonumber
&=\sum_{i=1}^{2n} \sum_{j=1}^{2n} c_i \t \beta_{ij} G_{i,j} - \sum_{i=1}^n c_i \left[ (\mu_i+\nu_i) S_i^* +\gamma R_i^* (\t r_i-1)\right] \left(1-\frac{1}{\t s}\right)(\t s_i-1).
\end{align}
Now, since it can be verified that over each directed cycle $\cC$ of $(\mathcal{G},\tilde{B})$, $\sum _{(s,r) \in \mathcal{E}(\mathcal{C})} G_{rs}=0$, by following the same arguments in the proof of \cite[Thm 3.5]{Shuai2013Lyapunov}, we have that
$\sum_{i=1}^{2n} \sum_{j=1}^{2n} c_i \t \beta_{ij} G_{i,j}=0$. Moreover, by assumption ${(\mu_i+\nu_i)}{S_i^*} \geq {\gamma R_i^*}$, for each $i=1, \ldots, n$, hence 
$$(\mu_i+\nu_i) S_i^* +\gamma R_i^* (r_i-1) \geq (\mu_i+\nu_i) S_i^* -\gamma R_i^* \geq 0, \qquad i=1, \ldots, n.$$
Thus, we have $\frac{dV}{dt} \leq 0$.
Since the largest compact invariant set in which $\frac{dV}{dt}=0$ is the singleton $\{x^* \}$, by LaSalle invariance principle our claim follows.
\end{proof}

\begin{remark}
 Note that if $\nu_i=0$ for all $i$, we obtain the same sufficient conditions for the GAS of the EE found for the SIRS model in \cite{Muroya2013}.
\end{remark}

\section{Numerical analysis}\label{simulations}

In this section, we explore the role of the network structures in the evolution of the epidemics. The primary criterion for parameter selection is the clarity of the resulting plot. Hence, the simulations were carried out with a set of parameters considered in \cite{ottaviano2022global}. These parameters, summarized in Table \ref{tab:param}, ensure that $\mathcal{R}_0>1$ in all the networks we consider, whose shapes are represented in Figure \ref{fig:shapes}.
\begin{table}[H]
\centering
\begin{tabular}{|c|c|c|c|c|c|c|c|c|c|}
\hline
$\beta^A_{ii}$ & $\beta^A_{ij}$ & $\beta^I_{ii}$ & $\beta^I_{ij}$ & $\mu_i$                         & $\nu_i$  & $\gamma$ & $\delta_A$ & $\delta_I$ & $\alpha$ \\ \hline
$0.8$          & $0.4$          & $0.95$         & $0.475$        & $1/(70\cdot365)$ & $0.01$ & $0.02$   & $0.1$      & $0.51$     & $0.8$    \\ \hline
\end{tabular}
\caption{Values of the parameters used in the simulations:
$\beta^A_{ii}=0.8$, which we reduced to $\beta^A_{ij}=0.4$ if $i\neq j$, to model a lower inter-community spreading; $\beta^I_{ii}=0.95$ and $\beta^I_{ij}=0.475$ if $i\neq j$; $\mu_i=1/(70\cdot365)$, meaning an average lifespan of $70$ years for all $i$; $\nu_i=0.01$, meaning $1\%$ of the susceptible population is vaccinated every day for all $i$; $\gamma=0.02$, meaning an average immunity of $50$ days; $\delta_A =0.1$, $\delta_I=0.51$, $\alpha=0.8$.\label{tab:param}}
\end{table}

In particular, we remark on how sensitive $\mathcal{R}_0$ is on the topology of the network, which is reflected in its adjacency matrix. 
Indeed, let us consider \eqref{R0_net} and let 
$$\beta_1=\min_{i,j} (M_1)_{i,j}, \qquad \text{and} \qquad \beta_2=\max_{i,j} (M_1)_{i,j}.$$
 Let us define $\bar \cA=\cA + I_n$, where $\cA$ is the adjacency matrix and $n$ the number of nodes of the network we are considering, respectively. Then, as a consequence of the Perron-Frobenius theorem, the following lower and upper bounds for $\cR_0$ hold:
\begin{equation}
\beta_1 \rho(\bar \cA) \leq \cR_0 \leq  \beta_2 \rho(\bar \cA)   
\end{equation}
In the case of the cycle-tree network in Figure \ref{fig:shapes}(a), we have $\rho(\cA)=3.2877$, for the star network in Figure \ref{fig:shapes}(b), $\rho(\cA)=3.8284$, in the case of the ring network in Figure \ref{fig:shapes}(c), $\rho(\cA)=3$, and for the line network in Figure \ref{fig:shapes}(d) we have $\rho(\cA)=2.9021$.
Consequently, in the star network, 
we found the largest $\mathcal{R}_0 = 4.91$, for the cycle-tree network we have $\mathcal{R}_0=4.37$. 
In the other two networks, i.e. the ring and the line, we find $\mathcal{R}_0= 4.07$ and $\mathcal{R}_0= 3.97$, respectively; 
we can see that the presence of one additional link in the ring increases the spectral radius of the transmission matrices and thus facilitates the spread of the disease.

We provide numerical simulations of the evolution of an epidemics 
 for the different 
 networks considered, showing the dynamics of the fraction of asymptomatic and symptomatic infected individuals (see Figures \ref{fig:real}, \ref{fig:stella}, \ref{fig:anello} and \ref{fig:strip}). In each simulations, in order to depict a realistic scenario, the epidemics starts in only one of the communities (Community 1), with a small asymptomatic fraction of the population and no symptomatic individuals, while the rest of the population is entirely susceptible. We obtain a delay in the start of the epidemics, directly proportional to the path distance of any community from Community 1: this is particularly visible in Figure \ref{fig:strip}. We observe a delay in the time of the peak, as well, although this is often less pronounced; this is clear in in Figure \ref{fig:anello}, in which communities with the same distance (path length) from Community 1 reach the peak at the same time. We underline that the time and the magnitude of the peak are directly proportional to the number of links of each community. Indeed, we can see that in the star network the peak of the non-central communities happens exactly at the same time and has the same magnitude, as one would expect, see Figure \ref{fig:stella}. In the ring network (Figure \ref{fig:anello}), all the communities only have two links, thus the time and the magnitude of the peak are the same for communities at the same path distance from Community 1. In Figure \ref{fig:strip}, instead, the magnitude is the same for Communities 2-8 and is lower in Communities 1 and 9, which are less connected with the others. The peak is reached faster in Community 1, in which the epidemic starts, and occurs later in Community 9, since it is the further and the less connected of the network. We remark that this predictable behaviour of the peak of infected individuals is due to the deterministic nature of the model, and thus of the numerical simulations.

For ease of interpretation, we plot the total number of Asymptomatic infected individuals and symptomatic Infected individuals in all four cases, see Figure \ref{fig:total}. The qualitative behaviour of all simulations is the same: after a first spike, the dynamics converges towards the endemic equilibrium, through quickly damping oscillations. In all our simulations, the endemic equilibrium values of $I$ are greater than the ones of $A$,
as we expected from \eqref{IofA} and our choice of the parameters involved in the formula.

Notice the significantly lower peaks in Figure \ref{fig:strip_tot}, when compared to \ref{fig:anello_tot}, even though the corresponding networks only differ for one edge, connecting Community 9 to Community 1, in which the epidemics start. The tables \ref{tab:treeI}-\ref{tab:strip} in Appendix \ref{appendix2} show the times in which the epidemic starts in each community, as well as the magnitude and the times of the peaks, both for asymptomatic ($A$) and symptomatic ($I$) infected individuals, for all the networks under investigations.

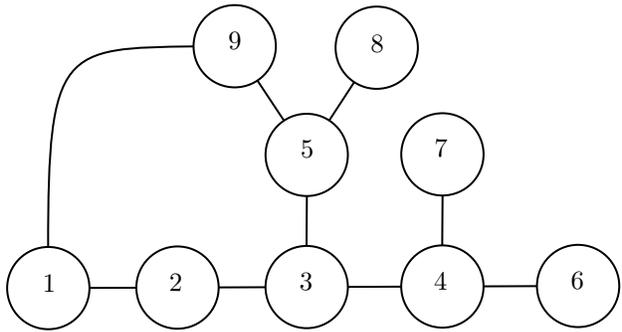
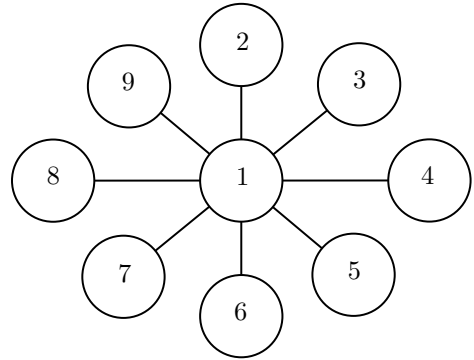
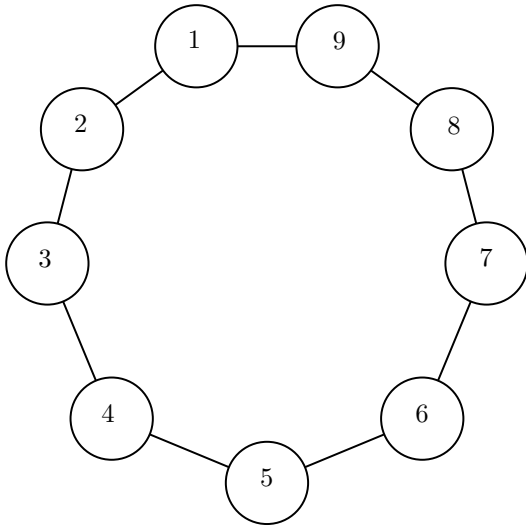
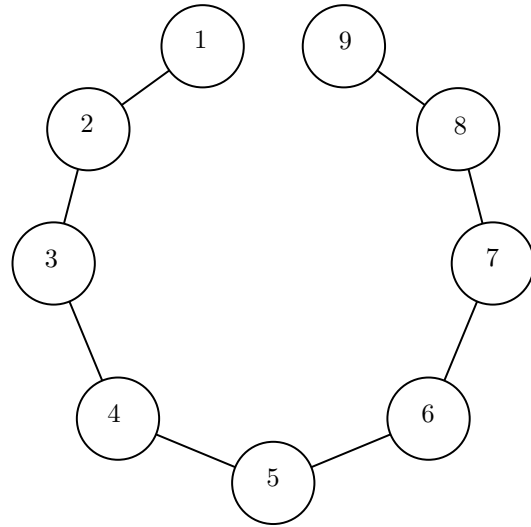
\begin{figure}[H]
    \centering
    \begin{subfigure}[t]{.49\textwidth}
      \centering

\tikzset{every picture/.style={line width=0.75pt}} 

\begin{tikzpicture}[x=0.70pt,y=0.70pt,yscale=-1,xscale=1]

\draw    (326.91,161.48) -- (364.74,103.51) ;
\draw    (326.91,161.48) -- (288,102.71) ;
\draw   (326.99,255.18) .. controls (314.71,255.23) and (304.71,245.31) .. (304.66,233.03) .. controls (304.61,220.75) and (314.53,210.75) .. (326.81,210.7) .. controls (339.09,210.65) and (349.09,220.57) .. (349.14,232.85) .. controls (349.19,245.13) and (339.27,255.13) .. (326.99,255.18) -- cycle ;
\draw   (400.32,254.88) .. controls (388.04,254.93) and (378.04,245.02) .. (377.99,232.73) .. controls (377.94,220.45) and (387.86,210.46) .. (400.14,210.41) .. controls (412.42,210.36) and (422.42,220.27) .. (422.47,232.56) .. controls (422.52,244.84) and (412.6,254.83) .. (400.32,254.88) -- cycle ;
\draw   (473.65,254.59) .. controls (461.37,254.64) and (451.37,244.72) .. (451.33,232.44) .. controls (451.28,220.16) and (461.19,210.16) .. (473.47,210.11) .. controls (485.76,210.06) and (495.75,219.98) .. (495.8,232.26) .. controls (495.85,244.54) and (485.93,254.54) .. (473.65,254.59) -- cycle ;
\draw   (400.33,183.42) .. controls (388.05,183.47) and (378.05,173.55) .. (378.01,161.27) .. controls (377.96,148.99) and (387.87,138.99) .. (400.15,138.94) .. controls (412.44,138.89) and (422.43,148.81) .. (422.48,161.09) .. controls (422.53,173.37) and (412.61,183.37) .. (400.33,183.42) -- cycle ;
\draw  [fill={rgb, 255:red, 255; green, 255; blue, 255 }  ,fill opacity=1 ] (288.09,124.95) .. controls (275.81,125) and (265.81,115.09) .. (265.76,102.8) .. controls (265.71,90.52) and (275.63,80.53) .. (287.91,80.48) .. controls (300.19,80.43) and (310.19,90.34) .. (310.24,102.63) .. controls (310.29,114.91) and (300.37,124.9) .. (288.09,124.95) -- cycle ;
\draw  [fill={rgb, 255:red, 255; green, 255; blue, 255 }  ,fill opacity=1 ] (364.83,125.75) .. controls (352.55,125.8) and (342.56,115.88) .. (342.51,103.6) .. controls (342.46,91.32) and (352.37,81.32) .. (364.66,81.27) .. controls (376.94,81.23) and (386.93,91.14) .. (386.98,103.42) .. controls (387.03,115.71) and (377.12,125.7) .. (364.83,125.75) -- cycle ;
\draw    (349.14,232.85) -- (377.99,232.73) ;
\draw    (422.47,232.56) -- (451.33,232.44) ;
\draw    (265.76,102.8) .. controls (195,103.71) and (187,103.71) .. (187.19,211.26) ;
\draw   (187.37,255.74) .. controls (175.09,255.79) and (165.09,245.87) .. (165.04,233.59) .. controls (164.99,221.31) and (174.91,211.31) .. (187.19,211.26) .. controls (199.47,211.21) and (209.47,221.13) .. (209.52,233.41) .. controls (209.57,245.69) and (199.65,255.69) .. (187.37,255.74) -- cycle ;
\draw   (257.18,255.46) .. controls (244.9,255.51) and (234.9,245.59) .. (234.85,233.31) .. controls (234.8,221.03) and (244.72,211.03) .. (257,210.98) .. controls (269.28,210.93) and (279.28,220.85) .. (279.33,233.13) .. controls (279.38,245.41) and (269.46,255.41) .. (257.18,255.46) -- cycle ;
\draw    (279.33,233.13) -- (304.66,233.03) ;
\draw    (209.52,233.41) -- (234.85,233.31) ;
\draw  [fill={rgb, 255:red, 255; green, 255; blue, 255 }  ,fill opacity=1 ] (327,183.71) .. controls (314.72,183.76) and (304.72,173.85) .. (304.67,161.57) .. controls (304.62,149.28) and (314.54,139.29) .. (326.82,139.24) .. controls (339.1,139.19) and (349.1,149.11) .. (349.15,161.39) .. controls (349.2,173.67) and (339.28,183.66) .. (327,183.71) -- cycle ;
\draw    (326.81,210.7) -- (327,183.71) ;
\draw    (400.14,210.41) -- (400.33,183.42) ;

\draw (321.62,223.71) node [anchor=north west][inner sep=0.75pt]  [rotate=-1.64] [align=left] {3};
\draw (393.99,223.63) node [anchor=north west][inner sep=0.75pt]  [rotate=-359.48] [align=left] {4};
\draw (321.89,152.27) node [anchor=north west][inner sep=0.75pt]  [rotate=-359.26] [align=left] {5};
\draw (468.22,223.18) node [anchor=north west][inner sep=0.75pt]  [rotate=-1.11] [align=left] {6};
\draw (394.33,151.58) node [anchor=north west][inner sep=0.75pt]  [rotate=-0.02] [align=left] {7};
\draw (359.89,95.08) node [anchor=north west][inner sep=0.75pt]  [rotate=-359.86] [align=left] {8};
\draw (282.86,93.74) node [anchor=north west][inner sep=0.75pt]  [rotate=-359.56] [align=left] {9};
\draw (251.23,224.05) node [anchor=north west][inner sep=0.75pt]  [rotate=-1.21] [align=left] {2};
\draw (182.79,224.4) node [anchor=north west][inner sep=0.75pt]  [rotate=-0.47] [align=left] {1};

\end{tikzpicture}
\caption{Cycle-tree network, i.e. a tree graph in which we add a cycle linking the first and the last community.}\label{fig:cycle-tree-net}
\end{subfigure}
   \begin{subfigure}[t]{.49\textwidth}
      \centering

\tikzset{every picture/.style={line width=0.75pt}} 

\begin{tikzpicture}[x=0.70pt,y=0.70pt,yscale=-1,xscale=1]

\draw    (390.43,204.62) -- (329.76,153.67) ;
\draw    (266.1,205.71) -- (329.76,153.67) ;
\draw    (329.76,153.67) -- (269.1,102.71) ;
\draw   (307.52,227) .. controls (307.52,214.72) and (317.48,204.76) .. (329.76,204.76) .. controls (342.04,204.76) and (352,214.72) .. (352,227) .. controls (352,239.28) and (342.04,249.24) .. (329.76,249.24) .. controls (317.48,249.24) and (307.52,239.28) .. (307.52,227) -- cycle ;
\draw    (329.76,175.9) -- (329.76,204.76) ;
\draw   (307.52,80.33) .. controls (307.52,68.05) and (317.48,58.1) .. (329.76,58.1) .. controls (342.04,58.1) and (352,68.05) .. (352,80.33) .. controls (352,92.62) and (342.04,102.57) .. (329.76,102.57) .. controls (317.48,102.57) and (307.52,92.62) .. (307.52,80.33) -- cycle ;
\draw    (329.76,102.57) -- (329.76,131.43) ;
\draw   (409.19,153.62) .. controls (409.19,141.34) and (419.15,131.38) .. (431.43,131.38) .. controls (443.71,131.38) and (453.67,141.34) .. (453.67,153.62) .. controls (453.67,165.9) and (443.71,175.86) .. (431.43,175.86) .. controls (419.15,175.86) and (409.19,165.9) .. (409.19,153.62) -- cycle ;
\draw   (205.86,153.71) .. controls (205.86,141.43) and (215.81,131.48) .. (228.1,131.48) .. controls (240.38,131.48) and (250.33,141.43) .. (250.33,153.71) .. controls (250.33,166) and (240.38,175.95) .. (228.1,175.95) .. controls (215.81,175.95) and (205.86,166) .. (205.86,153.71) -- cycle ;
\draw    (352,153.67) -- (409.19,153.62) ;
\draw    (250.33,153.71) -- (307.52,153.67) ;
\draw    (329.76,153.67) -- (393.43,101.62) ;
\draw  [fill={rgb, 255:red, 255; green, 255; blue, 255 }  ,fill opacity=1 ] (307.52,153.67) .. controls (307.52,141.38) and (317.48,131.43) .. (329.76,131.43) .. controls (342.04,131.43) and (352,141.38) .. (352,153.67) .. controls (352,165.95) and (342.04,175.9) .. (329.76,175.9) .. controls (317.48,175.9) and (307.52,165.95) .. (307.52,153.67) -- cycle ;
\draw  [fill={rgb, 255:red, 255; green, 255; blue, 255 }  ,fill opacity=1 ] (371.19,101.62) .. controls (371.19,89.34) and (381.15,79.38) .. (393.43,79.38) .. controls (405.71,79.38) and (415.67,89.34) .. (415.67,101.62) .. controls (415.67,113.9) and (405.71,123.86) .. (393.43,123.86) .. controls (381.15,123.86) and (371.19,113.9) .. (371.19,101.62) -- cycle ;
\draw  [fill={rgb, 255:red, 255; green, 255; blue, 255 }  ,fill opacity=1 ] (243.86,205.71) .. controls (243.86,193.43) and (253.81,183.48) .. (266.1,183.48) .. controls (278.38,183.48) and (288.33,193.43) .. (288.33,205.71) .. controls (288.33,218) and (278.38,227.95) .. (266.1,227.95) .. controls (253.81,227.95) and (243.86,218) .. (243.86,205.71) -- cycle ;
\draw  [fill={rgb, 255:red, 255; green, 255; blue, 255 }  ,fill opacity=1 ] (368.19,204.62) .. controls (368.19,192.34) and (378.15,182.38) .. (390.43,182.38) .. controls (402.71,182.38) and (412.67,192.34) .. (412.67,204.62) .. controls (412.67,216.9) and (402.71,226.86) .. (390.43,226.86) .. controls (378.15,226.86) and (368.19,216.9) .. (368.19,204.62) -- cycle ;
\draw  [fill={rgb, 255:red, 255; green, 255; blue, 255 }  ,fill opacity=1 ] (246.86,102.71) .. controls (246.86,90.43) and (256.81,80.48) .. (269.1,80.48) .. controls (281.38,80.48) and (291.33,90.43) .. (291.33,102.71) .. controls (291.33,115) and (281.38,124.95) .. (269.1,124.95) .. controls (256.81,124.95) and (246.86,115) .. (246.86,102.71) -- cycle ;

\draw (325.33,145) node [anchor=north west][inner sep=0.75pt]   [align=left] {1};
\draw (425.33,144.67) node [anchor=north west][inner sep=0.75pt]   [align=left] {4};
\draw (325.33,72) node [anchor=north west][inner sep=0.75pt]   [align=left] {2};
\draw (388.67,93) node [anchor=north west][inner sep=0.75pt]   [align=left] {3};
\draw (385.76,196.33) node [anchor=north west][inner sep=0.75pt]   [align=left] {5};
\draw (324.33,218) node [anchor=north west][inner sep=0.75pt]   [align=left] {6};
\draw (261.67,197.67) node [anchor=north west][inner sep=0.75pt]   [align=left] {7};
\draw (223,144.67) node [anchor=north west][inner sep=0.75pt]   [align=left] {8};
\draw (263.67,93.67) node [anchor=north west][inner sep=0.75pt]   [align=left] {9};

\end{tikzpicture}

      \caption{Star network, in which a central community is linked to all the others, and no other connections are present.}\label{fig:star-net}
\end{subfigure}
 \begin{subfigure}[t]{.49\textwidth}
      \centering

\tikzset{every picture/.style={line width=0.75pt}} 

\begin{tikzpicture}[x=0.70pt,y=0.70pt,yscale=-1,xscale=1]

\draw    (211.43,142) -- (230.17,69.4) ;
\draw    (230.17,69.4) -- (291.93,24.53) ;
\draw    (211.43,142) -- (246.19,225.91) ;
\draw    (246.19,225.91) -- (330.1,260.67) ;
\draw    (330.1,260.67) -- (414.01,225.91) ;
\draw    (414.01,225.91) -- (448.76,142) ;
\draw    (430.02,69.4) -- (448.76,142) ;
\draw    (368.26,24.53) -- (430.02,69.4) ;
\draw    (291.93,24.53) -- (368.26,24.53) ;
\draw  [fill={rgb, 255:red, 255; green, 255; blue, 255 }  ,fill opacity=1 ] (307.86,260.67) .. controls (307.86,248.38) and (317.81,238.43) .. (330.1,238.43) .. controls (342.38,238.43) and (352.33,248.38) .. (352.33,260.67) .. controls (352.33,272.95) and (342.38,282.9) .. (330.1,282.9) .. controls (317.81,282.9) and (307.86,272.95) .. (307.86,260.67) -- cycle ;
\draw  [fill={rgb, 255:red, 255; green, 255; blue, 255 }  ,fill opacity=1 ] (269.69,24.53) .. controls (269.69,12.25) and (279.65,2.3) .. (291.93,2.3) .. controls (304.21,2.3) and (314.17,12.25) .. (314.17,24.53) .. controls (314.17,36.81) and (304.21,46.77) .. (291.93,46.77) .. controls (279.65,46.77) and (269.69,36.81) .. (269.69,24.53) -- cycle ;
\draw  [fill={rgb, 255:red, 255; green, 255; blue, 255 }  ,fill opacity=1 ] (426.52,142) .. controls (426.52,129.72) and (436.48,119.76) .. (448.76,119.76) .. controls (461.04,119.76) and (471,129.72) .. (471,142) .. controls (471,154.28) and (461.04,164.24) .. (448.76,164.24) .. controls (436.48,164.24) and (426.52,154.28) .. (426.52,142) -- cycle ;
\draw  [fill={rgb, 255:red, 255; green, 255; blue, 255 }  ,fill opacity=1 ] (189.19,142) .. controls (189.19,129.72) and (199.15,119.76) .. (211.43,119.76) .. controls (223.71,119.76) and (233.67,129.72) .. (233.67,142) .. controls (233.67,154.28) and (223.71,164.24) .. (211.43,164.24) .. controls (199.15,164.24) and (189.19,154.28) .. (189.19,142) -- cycle ;
\draw  [fill={rgb, 255:red, 255; green, 255; blue, 255 }  ,fill opacity=1 ] (407.78,69.4) .. controls (407.78,57.12) and (417.74,47.16) .. (430.02,47.16) .. controls (442.3,47.16) and (452.26,57.12) .. (452.26,69.4) .. controls (452.26,81.68) and (442.3,91.64) .. (430.02,91.64) .. controls (417.74,91.64) and (407.78,81.68) .. (407.78,69.4) -- cycle ;
\draw  [fill={rgb, 255:red, 255; green, 255; blue, 255 }  ,fill opacity=1 ] (346.02,24.53) .. controls (346.02,12.25) and (355.98,2.3) .. (368.26,2.3) .. controls (380.54,2.3) and (390.5,12.25) .. (390.5,24.53) .. controls (390.5,36.81) and (380.54,46.77) .. (368.26,46.77) .. controls (355.98,46.77) and (346.02,36.81) .. (346.02,24.53) -- cycle ;
\draw  [fill={rgb, 255:red, 255; green, 255; blue, 255 }  ,fill opacity=1 ] (223.95,225.91) .. controls (223.95,213.63) and (233.9,203.67) .. (246.19,203.67) .. controls (258.47,203.67) and (268.42,213.63) .. (268.42,225.91) .. controls (268.42,238.19) and (258.47,248.15) .. (246.19,248.15) .. controls (233.9,248.15) and (223.95,238.19) .. (223.95,225.91) -- cycle ;
\draw  [fill={rgb, 255:red, 255; green, 255; blue, 255 }  ,fill opacity=1 ] (391.77,225.91) .. controls (391.77,213.63) and (401.72,203.67) .. (414.01,203.67) .. controls (426.29,203.67) and (436.24,213.63) .. (436.24,225.91) .. controls (436.24,238.19) and (426.29,248.15) .. (414.01,248.15) .. controls (401.72,248.15) and (391.77,238.19) .. (391.77,225.91) -- cycle ;
\draw  [fill={rgb, 255:red, 255; green, 255; blue, 255 }  ,fill opacity=1 ] (207.93,69.4) .. controls (207.93,57.12) and (217.89,47.16) .. (230.17,47.16) .. controls (242.45,47.16) and (252.41,57.12) .. (252.41,69.4) .. controls (252.41,81.68) and (242.45,91.64) .. (230.17,91.64) .. controls (217.89,91.64) and (207.93,81.68) .. (207.93,69.4) -- cycle ;

\draw (285.67,14.67) node [anchor=north west][inner sep=0.75pt]   [align=left] {1};
\draw (239,217) node [anchor=north west][inner sep=0.75pt]   [align=left] {4};
\draw (224.33,60.33) node [anchor=north west][inner sep=0.75pt]   [align=left] {2};
\draw (205,133.33) node [anchor=north west][inner sep=0.75pt]   [align=left] {3};
\draw (324.76,251.33) node [anchor=north west][inner sep=0.75pt]   [align=left] {5};
\draw (408.67,217) node [anchor=north west][inner sep=0.75pt]   [align=left] {6};
\draw (443.33,132.67) node [anchor=north west][inner sep=0.75pt]   [align=left] {7};
\draw (426,61) node [anchor=north west][inner sep=0.75pt]   [align=left] {8};
\draw (364,15.33) node [anchor=north west][inner sep=0.75pt]   [align=left] {9};

\end{tikzpicture}
          \caption{Ring network, in which each community is linked with the previous and next.}\label{fig:ring-net}
\end{subfigure}
 \begin{subfigure}[t]{.49\textwidth}
      \centering

\tikzset{every picture/.style={line width=0.75pt}} 

\begin{tikzpicture}[x=0.70pt,y=0.70pt,yscale=-1,xscale=1]

\draw    (211.43,142) -- (230.17,69.4) ;
\draw    (230.17,69.4) -- (291.93,24.53) ;
\draw    (211.43,142) -- (246.19,225.91) ;
\draw    (246.19,225.91) -- (330.1,260.67) ;
\draw    (330.1,260.67) -- (414.01,225.91) ;
\draw    (414.01,225.91) -- (448.76,142) ;
\draw    (430.02,69.4) -- (448.76,142) ;
\draw    (368.26,24.53) -- (430.02,69.4) ;
\draw  [fill={rgb, 255:red, 255; green, 255; blue, 255 }  ,fill opacity=1 ] (307.86,260.67) .. controls (307.86,248.38) and (317.81,238.43) .. (330.1,238.43) .. controls (342.38,238.43) and (352.33,248.38) .. (352.33,260.67) .. controls (352.33,272.95) and (342.38,282.9) .. (330.1,282.9) .. controls (317.81,282.9) and (307.86,272.95) .. (307.86,260.67) -- cycle ;
\draw  [fill={rgb, 255:red, 255; green, 255; blue, 255 }  ,fill opacity=1 ] (269.69,24.53) .. controls (269.69,12.25) and (279.65,2.3) .. (291.93,2.3) .. controls (304.21,2.3) and (314.17,12.25) .. (314.17,24.53) .. controls (314.17,36.81) and (304.21,46.77) .. (291.93,46.77) .. controls (279.65,46.77) and (269.69,36.81) .. (269.69,24.53) -- cycle ;
\draw  [fill={rgb, 255:red, 255; green, 255; blue, 255 }  ,fill opacity=1 ] (426.52,142) .. controls (426.52,129.72) and (436.48,119.76) .. (448.76,119.76) .. controls (461.04,119.76) and (471,129.72) .. (471,142) .. controls (471,154.28) and (461.04,164.24) .. (448.76,164.24) .. controls (436.48,164.24) and (426.52,154.28) .. (426.52,142) -- cycle ;
\draw  [fill={rgb, 255:red, 255; green, 255; blue, 255 }  ,fill opacity=1 ] (189.19,142) .. controls (189.19,129.72) and (199.15,119.76) .. (211.43,119.76) .. controls (223.71,119.76) and (233.67,129.72) .. (233.67,142) .. controls (233.67,154.28) and (223.71,164.24) .. (211.43,164.24) .. controls (199.15,164.24) and (189.19,154.28) .. (189.19,142) -- cycle ;
\draw  [fill={rgb, 255:red, 255; green, 255; blue, 255 }  ,fill opacity=1 ] (407.78,69.4) .. controls (407.78,57.12) and (417.74,47.16) .. (430.02,47.16) .. controls (442.3,47.16) and (452.26,57.12) .. (452.26,69.4) .. controls (452.26,81.68) and (442.3,91.64) .. (430.02,91.64) .. controls (417.74,91.64) and (407.78,81.68) .. (407.78,69.4) -- cycle ;
\draw  [fill={rgb, 255:red, 255; green, 255; blue, 255 }  ,fill opacity=1 ] (346.02,24.53) .. controls (346.02,12.25) and (355.98,2.3) .. (368.26,2.3) .. controls (380.54,2.3) and (390.5,12.25) .. (390.5,24.53) .. controls (390.5,36.81) and (380.54,46.77) .. (368.26,46.77) .. controls (355.98,46.77) and (346.02,36.81) .. (346.02,24.53) -- cycle ;
\draw  [fill={rgb, 255:red, 255; green, 255; blue, 255 }  ,fill opacity=1 ] (223.95,225.91) .. controls (223.95,213.63) and (233.9,203.67) .. (246.19,203.67) .. controls (258.47,203.67) and (268.42,213.63) .. (268.42,225.91) .. controls (268.42,238.19) and (258.47,248.15) .. (246.19,248.15) .. controls (233.9,248.15) and (223.95,238.19) .. (223.95,225.91) -- cycle ;
\draw  [fill={rgb, 255:red, 255; green, 255; blue, 255 }  ,fill opacity=1 ] (391.77,225.91) .. controls (391.77,213.63) and (401.72,203.67) .. (414.01,203.67) .. controls (426.29,203.67) and (436.24,213.63) .. (436.24,225.91) .. controls (436.24,238.19) and (426.29,248.15) .. (414.01,248.15) .. controls (401.72,248.15) and (391.77,238.19) .. (391.77,225.91) -- cycle ;
\draw  [fill={rgb, 255:red, 255; green, 255; blue, 255 }  ,fill opacity=1 ] (207.93,69.4) .. controls (207.93,57.12) and (217.89,47.16) .. (230.17,47.16) .. controls (242.45,47.16) and (252.41,57.12) .. (252.41,69.4) .. controls (252.41,81.68) and (242.45,91.64) .. (230.17,91.64) .. controls (217.89,91.64) and (207.93,81.68) .. (207.93,69.4) -- cycle ;

\draw (285.67,14.67) node [anchor=north west][inner sep=0.75pt]   [align=left] {1};
\draw (239,217) node [anchor=north west][inner sep=0.75pt]   [align=left] {4};
\draw (224.33,60.33) node [anchor=north west][inner sep=0.75pt]   [align=left] {2};
\draw (205,133.33) node [anchor=north west][inner sep=0.75pt]   [align=left] {3};
\draw (324.76,251.33) node [anchor=north west][inner sep=0.75pt]   [align=left] {5};
\draw (408.67,217) node [anchor=north west][inner sep=0.75pt]   [align=left] {6};
\draw (443.33,132.67) node [anchor=north west][inner sep=0.75pt]   [align=left] {7};
\draw (426,61) node [anchor=north west][inner sep=0.75pt]   [align=left] {8};
\draw (364,15.33) node [anchor=north west][inner sep=0.75pt]   [align=left] {9};

\end{tikzpicture}
          \caption{Line network, i.e. the ring network in which we remove the link between the communities $1$ and $9$.}\label{fig:line-net}
\end{subfigure}
 \caption{The four different network structures we consider in our numerical simulations. Circles represent the communities, numbered from 1 to 9, corresponding to C1 to C9 in Figures \ref{fig:real}, \ref{fig:stella}, \ref{fig:anello} and \ref{fig:strip}. Lines represent the links between the various communities. We use lines instead of arrows, since all networks are considered as undirected.}
 \label{fig:shapes}
\end{figure}

\begin{figure}[H]
    \centering
    \includegraphics[width=0.85\textwidth]{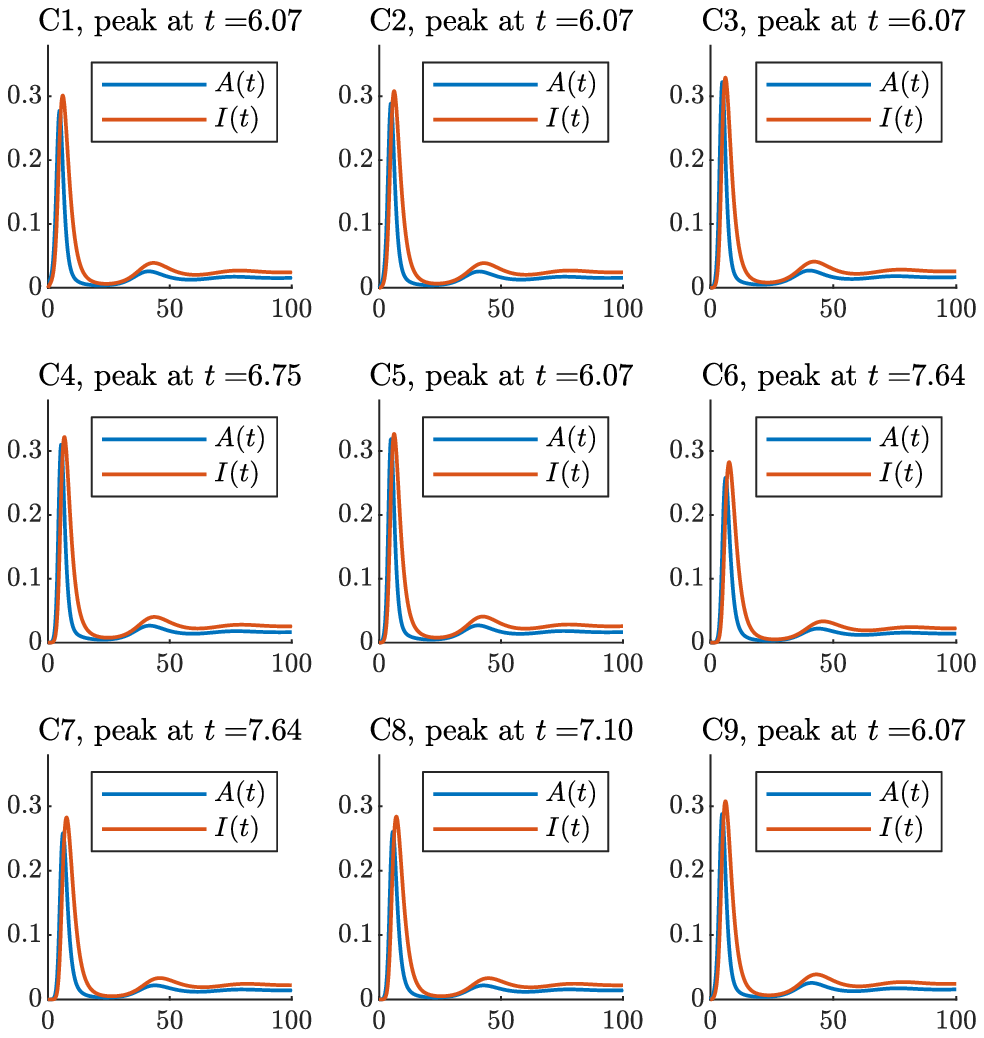}
    \caption{Evolution of the epidemic in each community of the cycle-tree network, see Figure \ref{fig:cycle-tree-net}. The title of each subplot indicates the community it represents, as well as the peak time of infected individuals. In this setting, from (\ref{R0_net}) we obtain $\mathcal{R}_0=4.37$. Refer to Table \ref{tab:param} for the values of the parameters.}
    \label{fig:real}
\end{figure}

\begin{figure}[H]
    \centering
    \includegraphics[width=0.85\textwidth]{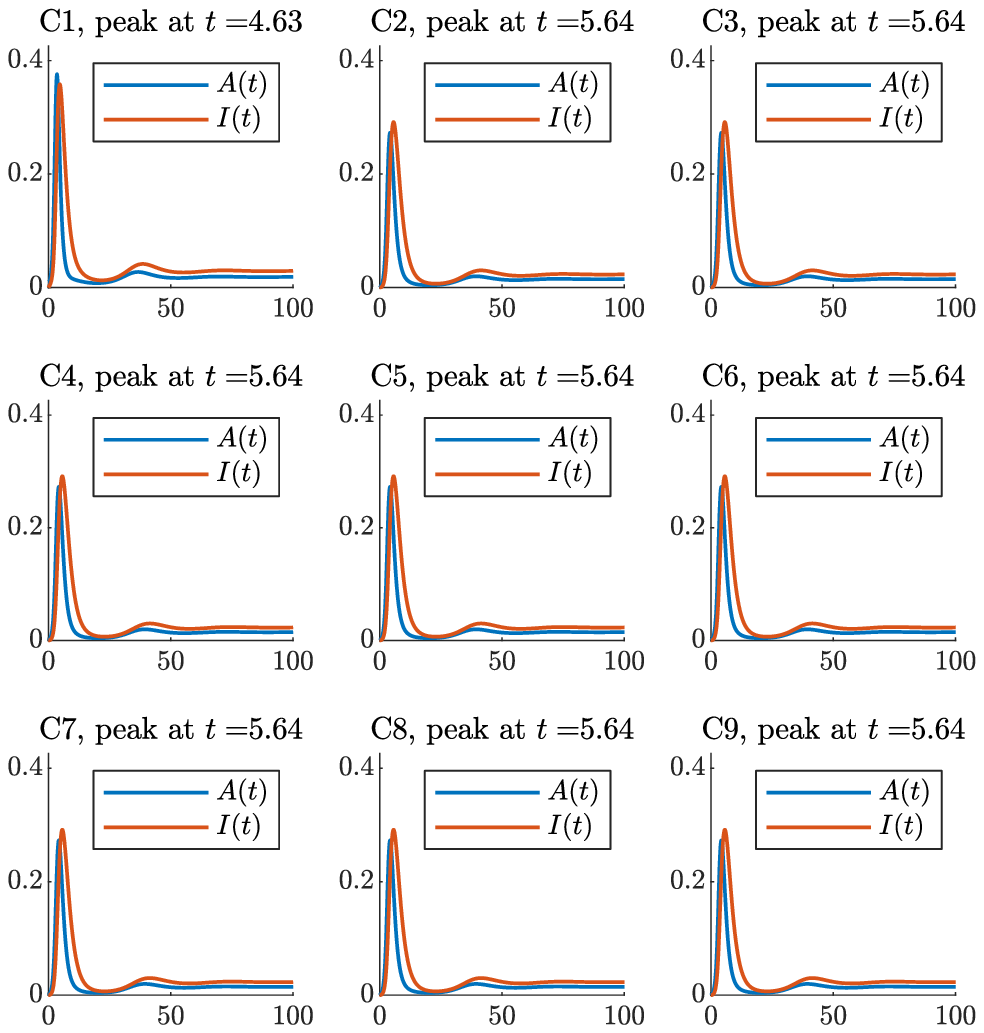}
    \caption{Evolution of the epidemic in each community of the star network, see Figure \ref{fig:star-net}. The title of each subplot indicates the community it represents, as well as the peak time of infected individuals. In this setting, from (\ref{R0_net}) we obtain $\mathcal{R}_0=4.91$. Refer to Table \ref{tab:param} for the values of the parameters.}
    \label{fig:stella}
\end{figure}

\begin{figure}[H]
    \centering
    \includegraphics[width=0.85\textwidth]{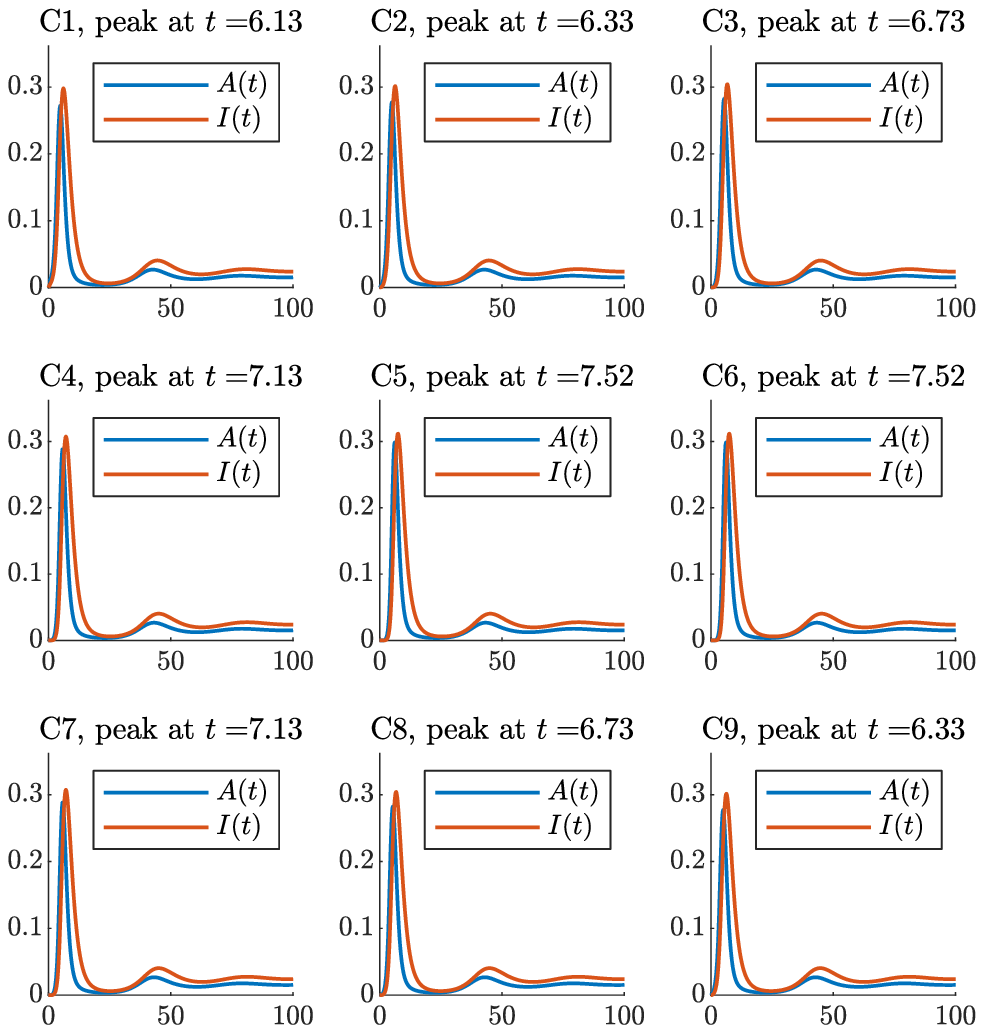}
    \caption{Evolution of the epidemic in each community of the ring network, see Figure \ref{fig:ring-net}. The title of each subplot indicates the community it represents, as well as the peak time of infected individuals. In this setting, from (\ref{R0_net}) we obtain $\mathcal{R}_0=4.07$. Refer to Table \ref{tab:param} for the values of the parameters.}
    \label{fig:anello}
\end{figure}

\begin{figure}[H]
    \centering
    \includegraphics[width=0.85\textwidth]{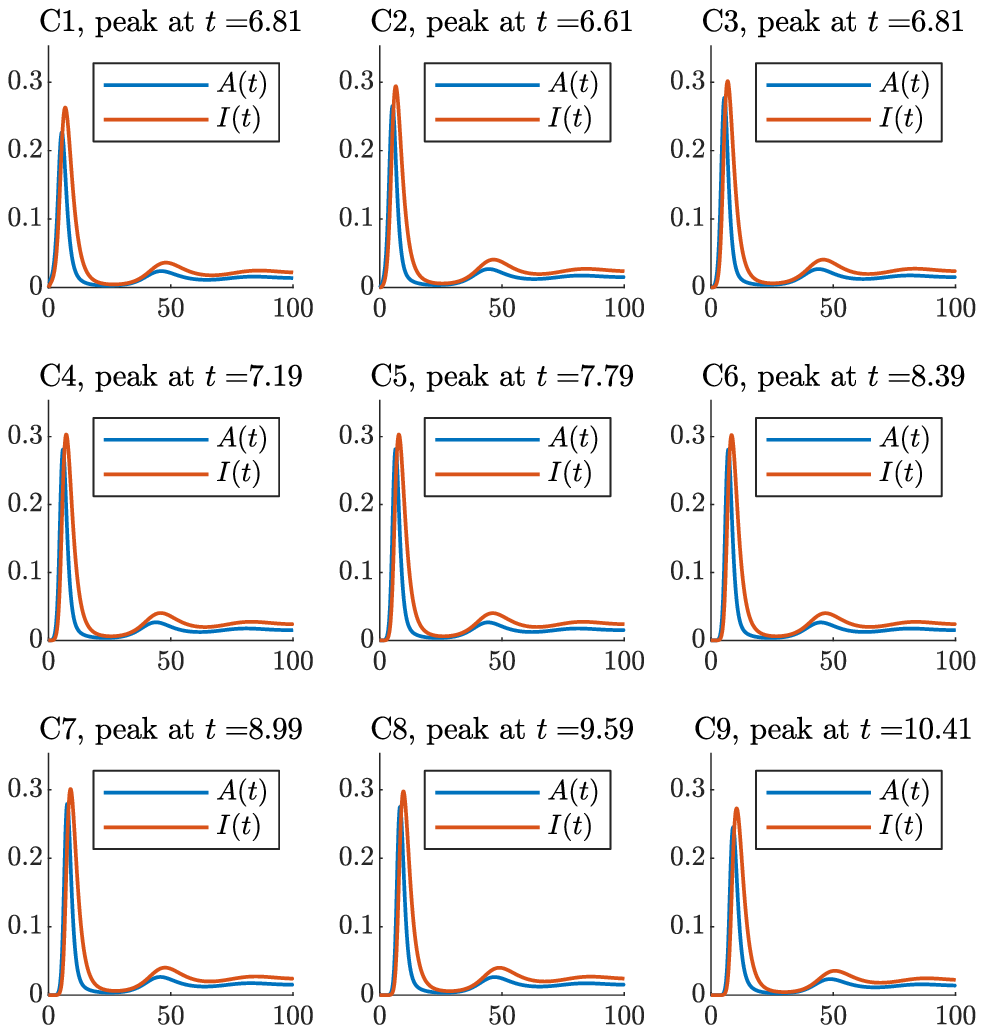}
    \caption{Evolution of the epidemic in each community of the line network, see Figure \ref{fig:line-net}. The title of each subplot indicates the community it represents, as well as the peak time of infected individuals. In this setting, from (\ref{R0_net}) we obtain $\mathcal{R}_0=3.97$. Refer to Table \ref{tab:param} for the values of the parameters.}
    \label{fig:strip}
\end{figure}

\begin{figure}[H]
    \begin{subfigure}{.49\textwidth}
        \centering
        \includegraphics[width=0.9\textwidth]{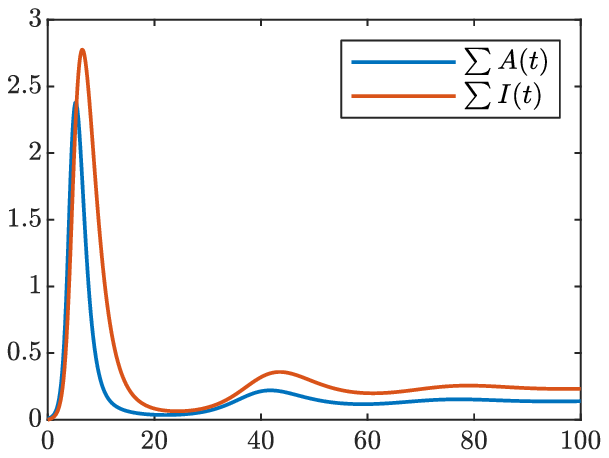}
        \caption{}\label{fig:real_tot}
    \end{subfigure}\hfill
    \begin{subfigure}{0.49\textwidth}
        \centering
        \includegraphics[width=0.9\textwidth]{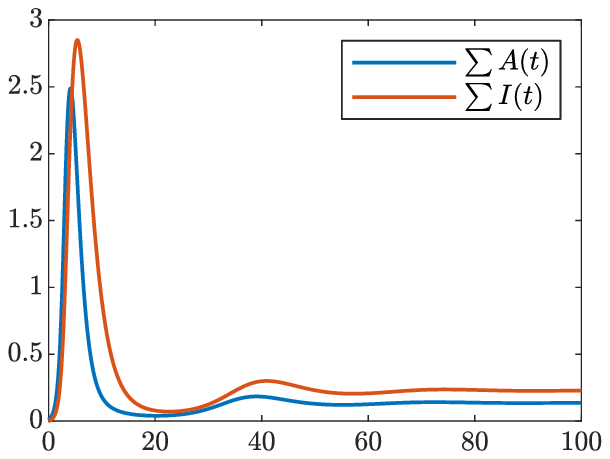}
        \caption{}\label{fig:stella_tot}
    \end{subfigure}
    \begin{subfigure}{0.49\textwidth}
        \centering
        \includegraphics[width=0.9\textwidth]{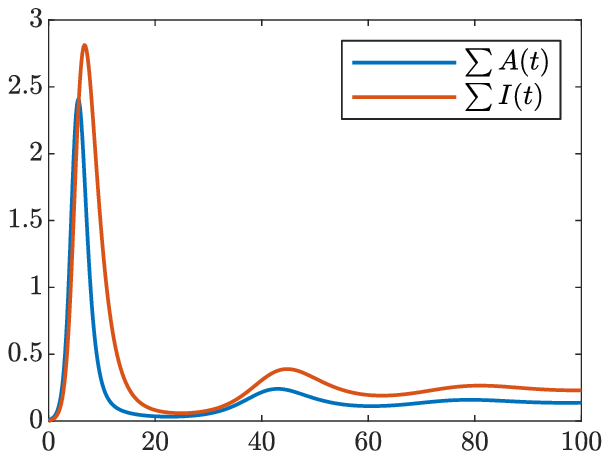}
        \caption{}\label{fig:anello_tot}
    \end{subfigure}\hfill
    \begin{subfigure}{0.49\textwidth}
        \centering
        \includegraphics[width=0.9\textwidth]{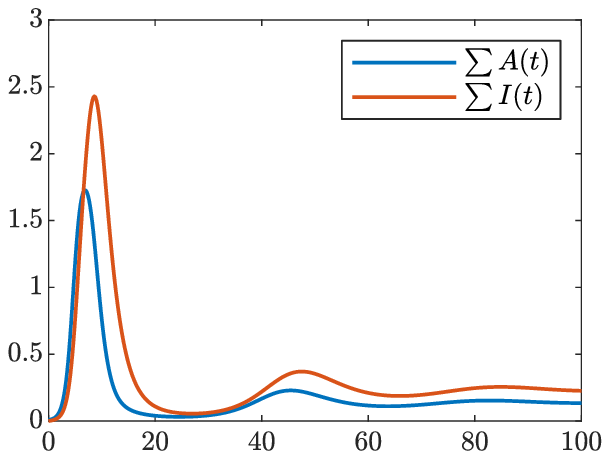}
        \caption{}\label{fig:strip_tot}
    \end{subfigure}
    \caption{Total amount of Asymptomatic infected ($\sum A(t)$) and symptomatic Infected ($\sum I(t)$) in the four networks we simulate. Respectively: (a) cycle-tree network, see Figure \ref{fig:cycle-tree-net}; (b) star network, see Figure \ref{fig:star-net}; (c) ring network, see Figure \ref{fig:ring-net}; and (d) line network, see Figure \ref{fig:line-net}. The qualitative behaviour is the same, i.e. convergence towards the endemic equilibrium through damped oscillations. Refer to Table \ref{tab:param} for the values of the parameters.}
\label{fig:total}
\end{figure}

\section{Conclusion}
We analysed a multi-group SAIRS-type epidemic model with vaccination. In this model, susceptible individuals can be infected by both asymptomatic and symptomatic infectious individuals, belonging to their communities as well as to other adjacent communities.

We provided a stability analysis of the
multi-group system under investigation; to the best of the authors' knowledge, this kind of analytical results was lacking in the literature.
\ste{Precisely}, we derived the expression of the basic reproduction number $\mathcal{R}_0$, which depends on the matrices which encode the transmission rates between and within communities.
We 
showed that if $\mathcal{R}_0 < 1$, the disease-free equilibrium is globally asymptotically stable, i.e. the disease will be eliminated in the long-run, whereas if $\mathcal{R}_0 > 1$ it is unstable. Moreover, in the SAIRS model without vaccination ($\nu_i= 0$, for all $i=1,\dots,n$), we were able to generalize the result on the global asymptotic stability of the disease-free equilibrium also in the case $\mathcal{R}_0 =1$. We 
proved the existence of a unique endemic equilibrium if $\mathcal{R}_0 >1$. We gave sufficient conditions for the local asymptotic stability of the endemic equilibrium; then, we investigated the global asymptotic stability of the endemic equilibrium in two cases. The first one regards the SAIR model (i.e. $\gamma = 0$), and does not requires any further conditions on the parameters besides $\mathcal{R}_0 >1$.

The second is the case of the SAIRS model, 
with the restriction that asymptomatic and symptomatic individuals have the same mean recovery period, i.e. $\delta_A = \delta_I$. In this case, we provided sufficient conditions for the GAS of the endemic equilibrium.

We leave as open problem the study of the global asymptotic stability of the endemic equilibrium for the SAIRS model with vaccination, in the case $\beta_A \neq \beta_I$ and $\delta_A \neq \delta_I$. Lastly, we conjecture that the conditions we derived to prove the asymptotic behaviour of the model are sufficient but not necessary conditions, as our numerical exploration of various settings seems to indicate. 

In this paper, we focused on a generalization of the SAIRS compartmental model proposed in \cite{ottaviano2022global}, by considering a network where each node represents a community; however, many others elements could be included in further generalizations to increase realism. For example, we may consider a greater number of compartments, e.g. including the “Exposed”, “Hospitalised” or “Quarantined” groups, or consider a nonlinear incidence rate; one could also introduce an additional disease-induced mortality, or an imperfect vaccination. We leave these as future research outlook.

\section*{Acknowledgments}
The authors would like to thank Prof.~Andrea Pugliese for the fruitful discussions, suggestions and careful reading of the paper draft.\\

The research of Stefania Ottaviano was supported by the University of Trento in the frame ``SBI-COVID - Squashing the business interruption curve while flattening pandemic curve (grant 40900013)''.\\

Mattia Sensi and Sara Sottile were supported by the Italian Ministry for University and Research (MUR) through the PRIN 2020 project ``Integrated Mathematical Approaches to Socio-Epidemiological Dynamics'' (No. 2020JLWP23).

\section*{Conflict of interest}

This work does not have any conflicts of interest.

\bibliographystyle{plain}
\bibliography{biblio_net}

\renewcommand{\thesection}{A}
\section{Appendix}\label{appendix}

\emph{Proof of Claim \ref{claimRe}}.
We recall that 
\begin{equation*}
\eta_i(z)=\frac{1}{\alpha+\mu_i+\delta_A}\left(z+\sum_{j=1}^n\beta^A_{ij}(A_j^*+h_j I_j^*)(1+K_i^1(z)+K_i^2(z)) \right).
\end{equation*}
It is easy to see that if $\Re(z)\geq 0$, then $\Re(K^1_i(z))>0$. Now, we show that if $\Re(z)\geq0$, then 
\begin{equation}\label{K2i}
 \Re(1+K^2_i(z))=\Re\left(1+\frac{1}{z+\mu_i+\nu_i+\gamma}\left((\delta_A-\nu_i)+\frac{(\delta_I-\nu_i)\alpha}{z+\delta_I+\nu_i}\right)\right)\geq0.
\end{equation}
For ease of notation, we define:
$$\eps=(\delta_I-\nu_i)\alpha, \qquad h_1=\mu_i+\nu_i+\gamma, \qquad h_2= \delta_I+\nu_i,$$
and let $z=a+ib$ in \eqref{K2i}.
If $\delta_A \geq \nu_i$, again it is easy to see that, if $\Re(z)\geq 0$, then
$$\Re\left(\frac{1}{z+h_1}(\delta_A-\nu_i)\right)\geq 0.$$
Now, let us show that 
\begin{equation}\label{re2}
 \Re\left(1+ \frac{\eps}{(z+h_1)(z+h_2)}\right) \geq 0.
\end{equation}
We have that
\begin{align}
\begin{split}
\Re\left(\frac{\eps}{(z+h_1)(z+h_2)}\right)&=\Re\left(\frac{\eps}{(a+h_1)(a+h_2)-b^2+ib (2a+h_1+h_2)}\right)\\
&= \Re\left(\eps\frac{ (a+h_1)(a+h_2)-b^2-ib (2a+h_1+h_2)}{((a+h_1)(a+h_2)-b^2)^2+b^2(2a+h_1+h_2)}\right)\\
&=\eps \frac{(a+h_1)(a+h_2)-b^2}{((a+h_1)(a+h_2)-b^2)^2+b^2(2a+h_1+h_2)}\\
&=\eps \frac{(P-b^2)}{(P-b^2)^2+b^2S^2}=g(b),
\end{split}
\end{align}
where we have introduced the notation
$$P=(a+h_1)(a+h_2) \qquad \text{and} \qquad S=(2a+h_1+h_2).$$
Since we assume $\delta_I \geq \nu_i$, we can see that the minimum of $g(b)$ is equal to $$\frac{-\eps}{2S\sqrt P+S^2},$$
and that
\begin{align*}
 \frac{-\eps}{2S\sqrt P+S^2} &\geq   \frac{-(\delta_I-\nu_i)\alpha}{2(2a+\mu_i+2\nu_i+\gamma+\delta_I)\sqrt{((a+\mu_i+\nu_i+\gamma)(a+\delta_I+\nu_i))}+(2a+\mu_i+2\nu_i+\gamma+\delta_I)^2}\\
&\geq \frac{-(\delta_I-\nu_i)\alpha}{2(\mu_i+2\nu_i+\gamma+\delta_I)\sqrt{((\mu_i+\nu_i+\gamma)(\delta_I+\nu_i))}+(\mu_i+2\nu_i+\gamma +\delta_I)^2}\\
&\geq -1.
\end{align*}
The last inequality holds since by hypothesis
$$(\delta_I-\nu_i)\alpha \leq 2(\mu_i+2\nu_i+\gamma+\delta_I)\sqrt{(\mu_i+\nu_i+\gamma)(\delta_I+\nu_i)}+(\mu_i+2\nu_i+\gamma +\delta_I)^2,$$
thus \eqref{re2} holds and the claim is proved.

\renewcommand{\thesection}{B}
\section{Appendix}\label{appendix2}

In the following tables, we show the times in which the epidemic starts in each community, as well as the magnitude and the times of the peaks, both for asymptomatic ($A$) and symptomatic ($I$) infected individuals, for all the networks under investigations. We consider an epidemic to have started in a community when the variable (either $I(t)$ or $A(t)$) exceeds the threshold value of $10^{-5}$. We remark that the quantity $I$, meaning the fraction of symptomatic individuals, is the one which is more realistically and accurately tracked, in a real-world scenario.

\begin{table}[H]
    \centering
    \begin{tabular}{c|c|c|c}
       \textbf{Community}  &   \textbf{Starting time of epidemic} &  \textbf{Time of peak} &  \textbf{Magnitude of peak}\\\hline 
    1  & 0 & 6.07 & 0.3011 \\
      2  & 0.006 &  6.07 & 0.3080\\
      3  & 0.006 & 6.07 & 0.3291 \\
      4 & 0.1    &  6.75& 0.3220\\
      5      &0.006 & 6.07 & 0.3268\\
      6       &0.65 & 7.64& 0.2826\\
      7        & 0.65 & 7.64& 0.2826\\
      8         & 0.17 & 7.10 & 0.2839\\       
      9         & 0.006 &6.07 &0.3077
    \end{tabular}
    \caption{Values of the starting time of the epidemic, time and magnitude of the peak in each community for the cycle-tree network in Fig. \ref{fig:shapes}(a) for symptomatic infected $I$. See also Figures \ref{fig:real} and \ref{fig:real_tot}.}
    \label{tab:treeI}
\end{table}

\begin{table}[H]
    \centering
    \begin{tabular}{c|c|c|c}
        \textbf{Community}  &   \textbf{Starting time of epidemic} &  \textbf{Time of peak} &  \textbf{Magnitude of peak}\\\hline 
    1  & 0 & 4.72 & 0.2773 \\
      2 & 0  &  4.72 & 0.2884\\
      3  & 0.005  & 4.93 & 0.3222 \\
      4   &0.12   &  5.33& 0.3098\\
      5   & 0.005  & 4.93 & 0.3188\\
      6       & 0.4 & 6.24& 0.2582\\
      7        & 0.4 & 6.24& 0.2582\\
      8         & 0.2& 5.54 & 0.2602\\       
      9         &0  &4.72&0.2878
    \end{tabular}
    \caption{Values of the starting time of the epidemic, time and magnitude of the peak in each community for the cycle-tree network in Fig. \ref{fig:shapes}(a) for asymptomatic infected $A$. See also Figures \ref{fig:real} and \ref{fig:real_tot}.}
    \label{tab:tree}
\end{table}

\begin{table}[H]
    \centering
    \begin{tabular}{c|c|c|c}
       \textbf{Community}  &   \textbf{Starting time of epidemic} &  \textbf{Time of peak} &  \textbf{Magnitude of peak}\\\hline 
     1  &0 & 4.63 & 0.3581 \\
      2    & 0.08 & 5.64& 0.2915\\
      3    & 0.08 & 5.64& 0.2915 \\
      4     & 0.08 & 5.64& 0.2915 \\
      5      & 0.08 & 5.64& 0.2915 \\
      6       &0.08  & 5.64& 0.2915 \\
      7        & 0.08 & 5.64& 0.2915 \\
      8         & 0.08 & 5.64& 0.2915 \\       
      9         &0.08  & 5.64& 0.2915
    \end{tabular}
    \caption{Values of the starting time of the epidemic, time and magnitude of the peak in each community for the star network in Fig. \ref{fig:shapes}(b) for symptomatic infected $I$. See also Figures \ref{fig:stella} and \ref{fig:stella_tot}.}
    \label{tab:starI}
\end{table}

\begin{table}[H]
    \centering
    \begin{tabular}{c|c|c|c}
        \textbf{Community}  &   \textbf{Starting time of epidemic} &  \textbf{Time of peak} &  \textbf{Magnitude of peak}\\\hline 
     1  &0 & 3.48 & 0.3759 \\
      2    & 0& 4.24& 0.2727\\
      3    & 0& 4.24& 0.2727\\
      4     & 0& 4.24& 0.2727\\
      5       & 0& 4.24& 0.2727\\
      6         & 0& 4.24& 0.2727\\
      7         & 0& 4.24& 0.2727\\
      8          & 0& 4.24& 0.2727\\ 
      9         & 0& 4.24& 0.2727\\
    \end{tabular}
    \caption{Values of the starting time of the epidemic, time and magnitude of the peak in each community for the star network in Fig. \ref{fig:shapes}(b) for asymptomatic infected $A$. See also Figures \ref{fig:stella} and \ref{fig:stella_tot}.}
    \label{tab:star}
\end{table}

\begin{table}[H]
    \centering
    \begin{tabular}{c|c|c|c}
        \textbf{Community}  &   \textbf{Starting time of epidemic} &  \textbf{Time of peak} &  \textbf{Magnitude of peak}\\\hline 
       1  &0 & 6.13 & 0.2981 \\
      2   &0.01 & 6.33 &  0.3015\\
      3    &0.4 &  6.73  &0.3041 \\
      4     &0.5 & 7.13& 0.3074 \\
      5      &1.7& 7.52 & 0.3117 \\
      6    &1.7& 7.52 & 0.3117 \\
      7     &0.5 & 7.13& 0.3074 \\
      8    &0.4 &  6.73  &0.3041 \\ 
      9     &0.01 & 6.33 &  0.3015\\
    \end{tabular}
    \caption{Values of the starting time of the epidemic, time and magnitude of the peak in each community for the ring network in Fig. \ref{fig:shapes}(c) for symptomatic infected $I$. See also Figures \ref{fig:anello} and \ref{fig:anello_tot}.}
    \label{tab:ringI}
\end{table}

\begin{table}[H]
    \centering
    \begin{tabular}{c|c|c|c}
        \textbf{Community}  &   \textbf{Starting time of epidemic} &  \textbf{Time of peak} &  \textbf{Magnitude of peak}\\\hline 
       1  &0 & 4.76 & 0.2718 \\
      2   &0 & 4.95 &  0.2771\\
      3    &0.15& 5.34  &0.2822 \\
      4    & 0.27 & 5.72 & 0.2882 \\
      5      &1.23 & 6.13 & 0.2985\\
      6    &1.23 & 6.13 & 0.2985\\
      7    & 0.27 & 5.72 & 0.2882 \\
      8   &0.15& 5.34  &0.2822 \\  
      9      &0 & 4.95 &  0.2771\\
    \end{tabular}
    \caption{Values of the starting time of the epidemic, time and magnitude of the peak in each community for the ring network in Fig. \ref{fig:shapes}(c) for asymptomatic infected $A$. See also Figures \ref{fig:anello} and \ref{fig:anello_tot}.}
    \label{tab:ring}
\end{table}

\begin{table}[H]
    \centering
    \begin{tabular}{c|c|c|c}
        \textbf{Community}  &   \textbf{Starting time of epidemic} &  \textbf{Time of peak} &  \textbf{Magnitude of peak}\\\hline 
   1  &0 & 6.81 &0.2629 \\
      2    & 0.02& 6.61 & 0.2941\\
      3    &0.2 & 6.81& 0.3015\\
      4     & 0.81& 7.19& 0.3029\\
      5      & 1.73& 7.79& 0.3031\\
      6       &1.88 & 8.39& 0.3023\\
      7        & 1.91& 8.99& 0.3011\\
      8         & 2.03& 9.59&0.2975 \\       
      9         &4.19 & 10.41&0.2728
    \end{tabular}
    \caption{Values of the starting time of the epidemic, time and magnitude of the peak in each community for the line network in Fig. \ref{fig:shapes}(d)  for symptomatic infected $I$. See also Figures \ref{fig:strip} and \ref{fig:strip_tot}.}
    \label{tab:stripI}
\end{table}

\begin{table}[H]
    \centering
    \begin{tabular}{c|c|c|c}
        \textbf{Community}  &   \textbf{Starting time of epidemic} &  \textbf{Time of peak} &  \textbf{Magnitude of peak}\\\hline 
   1  &0 &5.40 & 0.2263\\
      2    &0 & 5.26& 0.2650\\
      3    & 0.08& 5.54& 0.2771\\
      4     &0.53 & 5.96& 0.2811\\
      5      &1.26 & 6.42& 0.2814\\
      6       & 1.42&6.99 &0.2807 \\
      7        & 1.57& 7.58& 0.2791\\
      8         & 1.73& 8.39&0.2751 \\       
      9         & 3.64&8.99&0.2450
    \end{tabular}
    \caption{Values of the starting time of the epidemic, time and magnitude of the peak in each community for the line network in Fig. \ref{fig:shapes}(d)  for asymptomatic infected $A$. See also Figures \ref{fig:strip} and \ref{fig:strip_tot}.}
    \label{tab:strip}
\end{table}

\end{document}